\documentclass{article}%
\usepackage{amsfonts}
\usepackage{amssymb}
\usepackage{graphicx}
\usepackage{amsmath}%
\newtheorem{theorem}{Theorem}

\newtheorem{corollary}[theorem]{Corollary}

\newtheorem{definition}[theorem]{Definition}

\newtheorem{lemma}[theorem]{Lemma}

\newtheorem{proposition}[theorem]{Proposition}

\newenvironment{proof}[1][Proof]{\textbf{#1.} }{\ \rule{0.5em}{0.5em}}
\begin{document}

\title{The transition matroid\ of a 4-regular graph: an introduction}
\author{Lorenzo Traldi\\Lafayette College\\Easton, Pennsylvania 18042}
\date{}
\maketitle

\begin{abstract}
Given a 4-regular graph $F$, we introduce a binary matroid $M_{\tau}(F)$ on
the set of transitions of $F$. Parametrized versions of the Tutte polynomial
of $M_{\tau}(F)$ yield several well-known graph and knot polynomials,
including the Martin polynomial, the homflypt polynomial, the Kauffman
polynomial and the Bollob\'{a}s-Riordan polynomial.

\bigskip


\bigskip

Mathematics Subject\ Classification. 05C10, 05C31, 57M15, 57M25

\end{abstract}

\section{Introduction}

A graph is determined by two finite sets, one set containing vertices and the
other containing edges. Each edge is incident on one or two vertices; an edge
incident on only one vertex is a loop. We think of an edge as consisting of
two distinct half-edges, each of which is incident on precisely one vertex. In
this paper we are especially interested in 4-regular graphs, i.e., graphs in
which each vertex has precisely four incident half-edges. The special theory
of 4-regular graphs was initiated by Kotzig and although his definitions and
results have been generalized and modified over the years, most of the basic
ideas of the theory appear in his seminal paper \cite{K}.

Matroids were introduced by\ Whitney \cite{Wh}, and there are several standard
texts about them \cite{GM, O, We, W1, W2, W}. In this paper we will only
encounter binary matroids. If $M$ is a $GF(2)$-matrix with columns indexed by
the elements of a set $S$, then the binary matroid represented by $M$ is given
by defining the rank of each subset $A\subseteq S$ to be equal to the
dimension of the $GF(2)$-space spanned by the corresponding columns of $M$.
Matroids can be defined in many other ways. In particular, the minimal
nonempty subsets of $S$ that correspond to linearly dependent sets of columns
of $M$ are the circuits of the matroid represented by $M$. We will not refer
to matroid circuits often, to avoid confusion with the following definition.

A circuit in a graph is a sequence $v_{1}$, $h_{1}$, $h_{1}^{\prime}$, $v_{2}%
$, $h_{2}$, ..., $h_{k}$, $h_{k}^{\prime}=h_{0}^{\prime}$, $v_{k+1}=v_{1}$
such that for each $i\in\{1,...,k\}$, $h_{i}$ and $h_{i}^{\prime}$ are
half-edges of a single edge and $h_{i-1}^{\prime}$ and $h_{i}$ are both
incident on $v_{i}$. The half-edges that appear in a\ circuit must be pairwise
distinct, but vertices may be repeated. Two circuits are considered to be the
same if they differ only by a combination of cyclic permutations
$(1,...,k)\mapsto(i,...,k,1,...,i-1)$ and reversals $(1,...,k)\mapsto
(k,...,1)$. Notice that these definitions seem to be essentially
non-matroidal: circuits may be nested, and distinct circuits may involve
precisely the same vertices and half-edges, in different orders. For instance
if a graph has one vertex $v$ and two edges $e_{1}=\{f,f^{\prime}\}$ and
$e_{2}=\{h,h^{\prime}\}$ (both loops), then it has four different circuits:
$v$, $f$, $f^{\prime}$, $v$; $v$, $f$, $f^{\prime}$, $v$, $h^{\prime}$, $h$,
$v$; $v$, $f$, $f^{\prime}$, $v$, $h$, $h^{\prime}$, $v$ and $v$, $h$,
$h^{\prime}$, $v$. See Fig. \ref{tranmf2}, where these circuits are indicated
from left to right, using the convention that when a circuit traverses a
vertex, the dash style (dashed or undashed) is maintained.%

%TCIMACRO{\FRAME{ftbpFU}{4.5057in}{1.8568in}{0pt}{\Qcb{A 4-regular graph with
%one vertex has four distinct circuits.}}{\Qlb{tranmf2}}{tranmf2.ps}%
%width 4.5057in;  height 1.8568in;  depth 0pt;  original-width 8.4968in;
%original-height 11.0056in;  cropleft "0.2029";  croptop "0.8539";
%cropright "0.9055";  cropbottom "0.6322";
%filename 'tranmf2.ps';file-properties "XNPEU";}} }%
%BeginExpansion
\begin{figure}
[ptb]
\begin{center}
\includegraphics[
trim=1.724001in 6.957740in 0.802948in 1.607918in,
height=1.8568in,
width=4.5057in
]%
{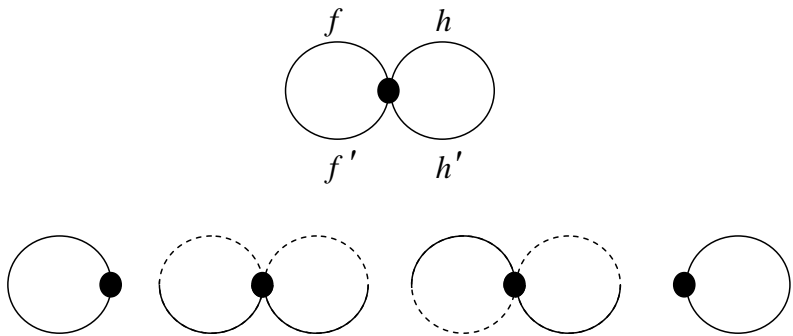}%
\caption{A 4-regular graph with one vertex has four distinct circuits.}%
\label{tranmf2}%
\end{center}
\end{figure}
%EndExpansion

A circuit $v_{1}$, $h_{1}$, $h_{1}^{\prime}$, $v_{2}$, ..., $h_{k}$,
$h_{k}^{\prime}=h_{0}^{\prime}$, $v_{k+1}=v_{1}$ in a 4-regular graph is
specified by the triples $h_{i-1}^{\prime},v_{i},h_{i}$ where $h_{i-1}%
^{\prime}$ and $h_{i}$ are distinct half-edges incident on $v_{i}$. We call
such a triple a \emph{single transition}. Kotzig called these triples
\textquotedblleft transitions\textquotedblright\ \cite{K}, but we adopt the
convention used by other authors (including Ellis-Monaghan and Sarmiento
\cite{ES}, Jaeger \cite{J} and Las Vergnas \cite{L1, L}) that a
\emph{transition} consists of two disjoint single transitions at the same vertex.

A circuit partition (or Eulerian partition or $\xi$-decomposition) of a
4-regular graph $F$ is a partition of $E(F)$ into edge-disjoint circuits.
These partitions were mentioned by Kotzig \cite{K}, and since then it has
become clear that they are of fundamental significance in the theory of
4-regular graphs. Expanding on earlier work of Martin \cite{Ma}, Las Vergnas
\cite{L} introduced the generating function that records the sizes of circuit
partitions of $F$, and also the generating functions that record the sizes of
directed circuit partitions of directed versions of $F$; he called these
generating functions the Martin polynomials of $F$. A circuit partition of $F$
is determined by choosing one of the three transitions at each vertex, and
Jaeger \cite{J} used this fact in defining his transition polynomial, a form
of the Martin polynomial that incorporates transition labels. A labeled form
of the Martin polynomial was independently discovered by Kauffman, who used it
in his bracket polynomial definition of the Jones polynomial of a knot or link
\cite{Kau, Kv}.

For plane graphs, there is an indirect connection between Martin polynomials
and graphic matroids, introduced by Martin \cite{Ma} and further elucidated by
Las Vergnas \cite{L2} and Jaeger \cite{J}. (The corresponding result for the
Kauffman bracket is due to Thistlethwaite \cite{Th}.) The complementary
regions of a 4-regular graph $F$ imbedded in the plane can be colored
checkerboard fashion, yielding a pair of dual graphs with $F$ as medial; the
cycle matroid of either of the two dual graphs yields the Martin polynomial of
a directed version of $F$. This indirect connection has been extended to
several formulas, each time weakening the connection with matroids: Jaeger
extended it to include information from the undirected Martin polynomial
\cite{J2}, Las Vergnas extended it to medial graphs in the projective plane
and the torus \cite{L1}, and Ellis-Monaghan and Moffatt extended it to include
medials in surfaces of all genera \cite{EMM, EMM1}.

The purpose of the present paper is to introduce a more general connection
between matroids and Martin polynomials, which holds for all 4-regular graphs
and does not require surface geometry. Let $F$ be a 4-regular graph, and let
$\mathfrak{T}(F)$ be the set of transitions of $F$. Then $\mathfrak{T}(F)$ is
partitioned into triples corresponding to the vertices of $F$; we call them
the \emph{vertex triples} of $\mathfrak{T}(F)$. Let $\mathcal{T}(F)$ be the
set of transversals of the vertex triples, i.e., subsets of $\mathfrak{T}(F)$
that contain exactly one transition for each vertex of $F$. Then there is a
bijection%
\[
\tau:\{\text{circuit partitions of }F\}\rightarrow\mathcal{T}(F)
\]
that assigns to a circuit partition $P$ the set of transitions involved in $P$.

\begin{theorem}
\label{tranmat}Let $F$ be a 4-regular graph with $c(F)$ connected components
and $n$ vertices. Then there is a rank-$n$ matroid $M_{\tau}(F)$ defined on
$\mathfrak{T}(F)$ whose rank function extends a well-known formula: For each
circuit partition $P$ of $F$, the rank of $\tau(P)$ in $M_{\tau}(F)$ is given
by
\[
r(\tau(P))=n+c(F)-\left\vert P\right\vert .
\]
Here $\left\vert P\right\vert $ denotes the size of $P$, i.e., the number of
circuits included in $P.$
\end{theorem}

We call $M_{\tau}(F)$ the \emph{transition matroid} of $F$, and we call the
equation displayed in Theorem \ref{tranmat} the \emph{circuit-nullity
formula}. (The name may seem more appropriate after the formula is rewritten
as $\left\vert P\right\vert -c(F)=n-r(\tau(P))$: the number of
\textquotedblleft extra\textquotedblright\ circuits in $P$ equals the nullity
of $\tau(P)$.) The formula has been rediscovered in one form or another many
times over the years \cite{Be, BM, Bu, Br, CL, J1, Jo, KR, Lau, MP, Me, M, R,
So, S, Tbn, Tnew, Z}. What is surprising about Theorem \ref{tranmat} is not
the circuit-nullity formula, but the fact that the formula can be extended to
give well-defined ranks for arbitrary subsets of $\mathfrak{T}(F)$, including
subsets not associated with circuit partitions in $F$ because they contain two
transitions at some vertex, or none.

We define $M_{\tau}(F)$ in Section 2. The definition is not concise enough to
summarize conveniently here, but we might mention that in the terminology of
\cite{Tnewnew}, $M_{\tau}(F)$ is the isotropic matroid of the interlacement
graph of any Euler system of $F$. Notice the word \textquotedblleft
any\textquotedblright\ in the preceding sentence: a typical 4-regular graph
has many different Euler systems, with nonisomorphic interlacement graphs; but
they all give rise to the same transition matroid.

In Section 3 we observe that the transition matroid of $F$ is closely related
to the $\Delta$-matroids and isotropic system associated to $F$ by Bouchet
\cite{Bdm, Bi1, Bi2, Bdr}.

\begin{theorem}
\label{adjiso}These statements about 4-regular graphs $F$ and $F^{\prime}$ are equivalent:

\begin{enumerate}
\item The transition matroid of $F$ is isomorphic to the transition matroid of
$F^{\prime}$.

\item The isotropic system of $F$ is isomorphic to the isotropic system of
$F^{\prime}$.

\item Any $\Delta$-matroid of $F$ is isomorphic to a $\Delta$-matroid of
$F^{\prime}$.

\item All $\Delta$-matroids of $F$ are isomorphic to $\Delta$-matroids of
$F^{\prime}$.

\item Any interlacement graph of $F$ is isomorphic to an interlacement graph
of $F^{\prime}$.

\item All interlacement graphs of $F$ are isomorphic to interlacement graphs
of $F^{\prime}$.
\end{enumerate}
\end{theorem}

Notice again the word \textquotedblleft any.\textquotedblright\ A typical
4-regular graph has many distinct interlacement graphs and many distinct
$\Delta$-matroids, but only one transition matroid and only one isotropic system.

Theorem \ref{adjiso} tells us that if an operation on 4-regular graphs
preserves interlacement graphs, then it also preserves transition matroids.
In\ Section 4 we discuss a theorem of Ghier \cite{gh}, a version of which was
subsequently discovered independently by Chmutov and Lando \cite{ChL}. Ghier's
theorem implies that there are two fundamental types of
interlacement-preserving operations, associated with edge cuts of size two or
four. Borrowing some terminology from knot theory, we call operations of the
first type connected sums; their inverses are separations. (See Fig.
\ref{tranmf5a}.) Operations of the second type are balanced mutations. (See
Fig. \ref{tranmf5b1}.) The formal definitions follow.

\begin{definition}
\label{connectedsum}Suppose $F_{1}$ and $F_{2}$ are two separate 4-regular
graphs. Let $e_{1}=\{h_{1},h_{1}^{\prime}\}\in E(F_{1})$ and $e_{2}%
=\{h_{2},h_{2}^{\prime}\}\in E(F_{2})$. Then a \emph{connected sum of }$F_{1}%
$\emph{\ and\ }$F_{2}$\emph{ with respect to }$e_{1}$\emph{\ and }$e_{2}$ is a
graph obtained from $F_{1}\cup F_{2}$ by replacing $e_{1}$ and $e_{2}$ with
new edges $\{h_{1},h_{2}\}$ and $\{h_{1}^{\prime},h_{2}^{\prime}\}$.
\end{definition}

Notice that there are two distinct connected sums of $F_{1}$ and $F_{2}$ with
respect to $e_{1}$ and $e_{2}$, which involve matching different pairs of
half-edges. Moreover there are many different connected sums of $F_{1}$ and
$F_{2}$, with different choices of $e_{1}$ and $e_{2}$. In contrast, the
inverse operation is uniquely defined.%
%TCIMACRO{\FRAME{ftbpFU}{4.2877in}{0.7178in}{0pt}{\Qcb{Connected sum and
%separation.}}{\Qlb{tranmf5a}}{tranmf5a.ps}%
%width 4.2877in;  height 0.7178in;  depth 0pt;  original-width 8.4968in;
%original-height 11.0056in;  cropleft "0.2047";  croptop "0.8601";
%cropright "0.8730";  cropbottom "0.7764";
%filename 'tranmf5a.ps';file-properties "XNPEU";}} }%
%BeginExpansion
\begin{figure}
[ptb]
\begin{center}
\includegraphics[
trim=1.739295in 8.544748in 1.079093in 1.539684in,
height=0.7178in,
width=4.2877in
]%
{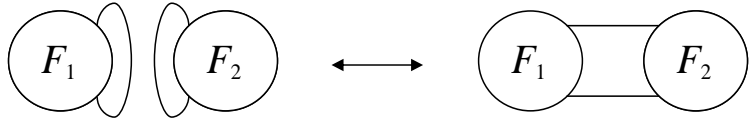}%
\caption{Connected sum and separation.}%
\label{tranmf5a}%
\end{center}
\end{figure}
%EndExpansion

\begin{definition}
\label{separation}Suppose two edges $e_{1}=\{h_{1},h_{1}^{\prime}\}$ and
$e_{2}=\{h_{2},h_{2}^{\prime}\}$ constitute an edge cut of $F$. Then the
\emph{separation of }$F$ \emph{with respect to }$e_{1}$ and $e_{2}$ is the
graph $F^{\prime}$ with $c(F^{\prime})=c(F)+1$ that is obtained from $F$ by
replacing $e_{1}$ and $e_{2}$ with new edges $\{h_{1},h_{2}\}$ and
$\{h_{1}^{\prime},h_{2}^{\prime}\}$.
\end{definition}

The second type of interlacement-preserving operation is more complicated.

\begin{definition}
\label{mutation}Suppose $F$ has nonempty subgraphs $F_{1}$ and $F_{2}$ such
that $V(F)=V(F_{1})\cup V(F_{2})$, $V(F_{1})\cap V(F_{2})=\varnothing$ and
precisely four edges of $F$ connect $F_{1}$ to $F_{2}$. Group these four edges
into two pairs, and reassemble the four half-edges from each pair into two new
edges, each of which connects $F_{1}$ to $F_{2}$. The resulting graph
$F^{\prime}$ is obtained from $F$ by \emph{balanced mutation}.
\end{definition}

The reader familiar with knot theory should find balanced mutation rather
strange, in at least three ways. 1. Knot theorists use diagrams in the plane,
but we are discussing abstract graphs, not plane graphs. In particular, the
apparent edge-crossings in Fig. \ref{tranmf5b1} are mere artifacts of the
figure; they do not reflect any significant information about the graphs
involved in the operation. 2. Similarly, knot-theoretic mutation is usually
depicted by rotating $F_{2}$ through an angle of $\pi$, rather than
repositioning half-edges. This rotation implicitly requires that new edges and
old edges correspond in pairs. As our discussion is abstract, we mention this
requirement explicitly. 3. In knot theory it is not required that the four
edges shown in the diagram be distinct; two might be parts of an arc that
simply passes through the region of the plane denoted $F_{1}$ or $F_{2}$,
without encountering any crossing.\ This third issue is not significant,
though, because such trivial knot-theoretic mutations can be accomplished with
separations and connected sums.%

%TCIMACRO{\FRAME{ftbpFU}{3.9894in}{0.9591in}{0pt}{\Qcb{Balanced mutation.}%
%}{\Qlb{tranmf5b1}}{tranmf5b1.ps}{\special{ language "Scientific Word";
%type "GRAPHIC";  maintain-aspect-ratio TRUE;  display "USEDEF";
%valid_file "F";  width 3.9894in;  height 0.9591in;  depth 0pt;
%original-width 8.4968in;  original-height 11.0056in;  cropleft "0.2361";
%croptop "0.8904";  cropright "0.8581";  cropbottom "0.7777";
%filename 'tranmf5b1.ps';file-properties "XNPEU";}} }%
%BeginExpansion
\begin{figure}
[ptb]
\begin{center}
\includegraphics[
trim=2.006095in 8.559055in 1.205696in 1.206214in,
height=0.9591in,
width=3.9894in
]%
{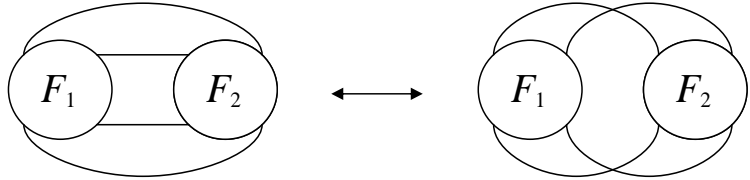}%
\caption{Balanced mutation.}%
\label{tranmf5b1}%
\end{center}
\end{figure}
%EndExpansion

Ghier's theorem \cite{gh} implies that Definitions \ref{connectedsum},
\ref{separation} and \ref{mutation} characterize the 4-regular graphs with
isomorphic transition matroids.

\begin{theorem}
\label{iso}Let $F$ and $F^{\prime}$ be 4-regular graphs. Then $M_{\tau
}(F)\cong M_{\tau}(F^{\prime})$ if and only if $F$ can be obtained from
$F^{\prime}$ using connected sums, separations and balanced mutations.
\end{theorem}

In particular, Theorem \ref{iso} tells us that if $F$ is a connected 4-regular
graph that admits neither a separation nor a nontrivial balanced mutation,
then $F$ is determined by its transition matroid:

\begin{corollary}
\label{isom}Let $F$ be a connected 4-regular graph in which every edge cut
with fewer than six edges consists of the edges incident on one vertex. Then a
4-regular graph $F^{\prime}$ is isomorphic to $F$ if and only if their
transition matroids are isomorphic.
\end{corollary}

Theorem \ref{iso} is the analogue for transition matroids of Whitney's famous
characterization of graphs with isomorphic cycle matroids (see Chapter 5 of
\cite{O}): connected sums are analogous to vertex identifications in disjoint
graphs, separations are analogous to cutpoint separations, and balanced
mutations are analogous to Whitney twists. Corollary \ref{isom} is the
analogue of a special case of Whitney's theorem: 3-connected simple graphs are
determined by their cycle matroids.

These analogies may be explained using an idea that appeared implicitly in
Jaeger's work \cite{J1} and was later discussed in detail by Bouchet
\cite{Bi1}. A circuit partition $P$ in a 4-regular graph $F$ has an associated
touch-graph $Tch(P)$, which has a vertex for each circuit of $P$ and an edge
for each vertex of $F$; the edge of $Tch(P)$ corresponding to $v$ is incident
on the vertex or vertices of $Tch(P)$ corresponding to circuit(s) of $P$ that
pass through $v$. In Section 5 we\ observe that connected sums, separations
and balanced mutations of 4-regular graphs induce vertex identifications,
cutpoint separations and Whitney twists in touch-graphs. We also show that
Theorem \ref{tranmat} implies the formula of Las Vergnas and Martin for plane
graphs that was mentioned above.

A parametrized Tutte polynomial of a matroid $M$ on a ground set $S$ is a sum
of the form
\[
f(M)=\sum_{A\subseteq S}\left(  \prod\limits_{a\in A}\alpha(a)\right)  \left(
\prod\limits_{s\in S-A}\beta(s)\right)  (x-1)^{r(S)-r(A)}(y-1)^{\left\vert
A\right\vert -r(A)}%
\]
where $r$ is the rank function of $M$ and $\alpha$, $\beta$ are weight
functions mapping $S$ into some commutative ring. (Other types of parametrized
Tutte polynomials appear in the literature \cite{BR, EMT, Za}, but they will
not enter into our discussion.) Different choices of $\alpha$ and $\beta$
produce parametrized Tutte polynomials with different amounts of information,
but if we are free to choose $\alpha$ and $\beta$ as we like, we can produce a
parametrized Tutte polynomial in which the products of parameters identify the
contributions of individual subsets. Such a polynomial determines the rank of
each subset of $S$, so it determines the matroid $M$. Consequently, an
equivalent form of Theorem \ref{iso} is the following.

\begin{theorem}
Let $F$ and $F^{\prime}$ be 4-regular graphs. Then $F$ can be obtained from
$F^{\prime}$ using connected sums, separations and balanced mutations if and
only if for all weight functions on $\mathfrak{T}(F)$, there are corresponding
weight functions on $\mathfrak{T}(F^{\prime})$ such that the resulting
parametrized Tutte polynomials of $M_{\tau}(F)$ and $M_{\tau}(F^{\prime})$ are
the same.
\end{theorem}

Theorem \ref{tranmat} provides a direct connection between Martin polynomials
and parametrized Tutte polynomials of transition matroids.

\begin{theorem}
\label{MTutte}If $F$ is a 4-regular graph then $c(F)$ and the parametrized
Tutte polynomial of $M_{\tau}(F)$ determine the Martin polynomial of $F$ and
the directed Martin polynomials of all balanced orientations of $F$.
\end{theorem}

Theorem \ref{MTutte} is surprising because it has been assumed that for
non-planar 4-regular graphs, Martin polynomials are not directly related to
Tutte polynomials of matroids. (As noted above, the circuit theory of
4-regular graphs seems to be fundamentally non-matroidal, because distinct
circuits can be built from the same vertices and half-edges.) The theorem also
implies that several other graph polynomials associated with circuit
partitions in 4-regular graphs can be recovered directly from the Tutte
polynomial of $M_{\tau}(F)$.

\begin{theorem}
\label{JTutte}If $F$ is a 4-regular graph then $c(F)$ and the parametrized
Tutte polynomial of $M_{\tau}(F)$ determine the interlace polynomials of all
circle graphs associated with $F$ \cite{AH, A1, A2, A, C, T6}, the transition
polynomial of $F$ \cite{E, J}, and the homflypt \cite{FYHLMO, PT} and Kauffman
\cite{KI} polynomials of all knots and links with diagrams whose underlying
4-regular graphs are isomorphic to $F$.
\end{theorem}

Another kind of structure connected with 4-regular graphs is a ribbon graph.
These objects may be defined in several different ways, and in Section 7 we
focus on one particular definition, according to which a ribbon graph is
constructed from a 4-regular graph given with two circuit partitions that do
not involve the same transition at any vertex. Using this definition it is not
hard to deduce the following.

\begin{theorem}
\label{BR}If $F$ is a 4-regular graph then $c(F)$ and the parametrized Tutte
polynomial of $M_{\tau}(F)$ determine all weighted Bollob\'{a}s-Riordan and
topological Tutte polynomials of ribbon graphs whose medial graphs are
isomorphic to $F$.
\end{theorem}

We hope that Theorems \ref{MTutte} -- \ref{BR} will provide both a unifying
theme for the many polynomials mentioned and useful applications of the theory
of para\-metrized Tutte polynomials. The representation of ribbon graphs using
pairs of circuit partitions also yields an embedding-free expression of the
theory of twisted duality due to Chmutov \cite{Ch}, Ellis-Monaghan and Moffatt
\cite{EMM}: two ribbon graphs are twisted duals if and only if they correspond
to pairs of circuit partitions in the same 4-regular graph.

In Section 9 we\ relate these ideas to a characterization of planar 4-regular
graphs that appears in discussions of the famous Gauss crossing problem (see
for instance \cite{DFOM, GR, RR}): a 4-regular graph is planar if and only if
it has an Euler system whose interlacement graph is bipartite.

We are grateful to Joanna Ellis-Monaghan for her encouragement and her
illuminating explanations regarding embedded graphs and their associated
polynomials, and to Sergei Chmutov and Iain Moffatt for their comments on
earlier versions of the paper. The final version of the paper also benefited
from comments and corrections offered by two anonymous readers.

\section{Defining the transition matroid}

Recall that an Euler circuit of a 4-regular graph $F$ is a circuit that
includes every half-edge of $F$. If $F$ is disconnected it cannot have an
Euler circuit, of course, but every 4-regular graph $F$ has an Euler system,
i.e., a set containing one Euler circuit for each connected component of $F$.
\textit{Kotzig's theorem} \cite{K} states that given one Euler system of $F$,
all the others are obtained by applying sequences of $\kappa$-transformations:
given an Euler system $C$ and a vertex $v$, the $\kappa$-transform $C\ast v$
is the Euler system obtained from $C$ by reversing one of the two $v$-to-$v$
walks within the circuit of $C$ incident on~$v$. See Fig. \ref{tranmf1}.

Let $C$ be an Euler system of a 4-regular graph $F$. The interlacement graph
$\mathcal{I}(C)$ is a simple graph with the same vertices as $F$. Two vertices
$v\neq w\in V(F)$ are adjacent in $\mathcal{I}(C)$ if and only if they appear
in the order $v...w...v...w$ on one of the circuits of $C$; if this happens
then $v$ and $w$ are interlaced with respect to $C$. Interlacement graphs were
discussed by Bouchet \cite{Bold} and Read and Rosenstiehl \cite{RR}, who
observed that the relationship between $\mathcal{I}(C)$ and $\mathcal{I}(C\ast
v)$ is described by simple local complementation at $v$: if $v\neq x\neq y\neq
v$ and $x,y$ are both neighbors of $v$ in $\mathcal{I}(C)$ then they are
adjacent in $\mathcal{I}(C\ast v)$ if and only if they are not adjacent in
$\mathcal{I}(C)$.%

%TCIMACRO{\FRAME{ftbFU}{2.9879in}{0.5846in}{0pt}{\Qcb{$C$ and $C\ast v$.}%
%}{\Qlb{tranmf1}}{tranmf1.ps}{\special{ language "Scientific Word";
%type "GRAPHIC";  maintain-aspect-ratio TRUE;  display "USEDEF";
%valid_file "F";  width 2.9879in;  height 0.5846in;  depth 0pt;
%original-width 8.4968in;  original-height 11.0056in;  cropleft "0.3150";
%croptop "0.8805";  cropright "0.7795";  cropbottom "0.8135";
%filename 'tranmf1.ps';file-properties "XNPEU";}} }%
%BeginExpansion
\begin{figure}
[tb]
\begin{center}
\includegraphics[
trim=2.676492in 8.953055in 1.873544in 1.315169in,
height=0.5846in,
width=2.9879in
]%
{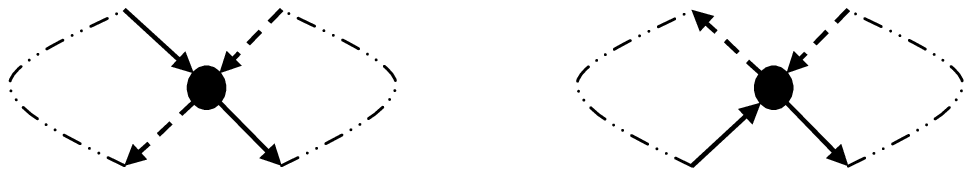}%
\caption{$C$ and $C\ast v$.}%
\label{tranmf1}%
\end{center}
\end{figure}
%EndExpansion

Recall that a transition at a vertex $v$ is a partition of the four half-edges
incident at a vertex $v$ into two pairs, and $\mathfrak{T}(F)$ denotes the set
of all transitions in $F$. For each Euler system $C$ of $F$ and each $v\in
V(F)$ we may label the transitions at $v$ according to their relationships
with $C$, following a notational scheme from \cite{Tbn, T5, T6, Tnew}.

\begin{definition}
The transition that is used by the circuit of $C$ incident at $v$ is labeled
$\phi_{C}(v)$, the other transition that is consistent with an orientation of
this circuit is labeled $\chi_{C}(v)$, and the transition that is not
consistent with an orientation of the incident circuit of $C$ is labeled
$\psi_{C}(v)$.
\end{definition}

It is easy to see that transition labels with respect to $C$ are not changed
when different orientations of the circuits of $C$ are used. However, a given
transition may have different labels with respect to different Euler systems.
The differences among these labels are not arbitrary:

\begin{proposition}
\label{labelinv}If $C$ is an Euler system of $F$ then the labels of elements
of $\mathfrak{T}(F)$ with respect to $C$ determine the labels of elements of
$\mathfrak{T}(F)$ with respect to all other Euler systems of $F$.
\end{proposition}

\begin{proof}
As illustrated in Fig. \ref{tranmf14}, a $\kappa$-transformation changes
transition labels in the following way: $\psi_{C\ast v}(v)=\phi_{C}(v)$ and
$\phi_{C\ast v}(v)=\psi_{C}(v)$; if $w$ is a neighbor of $v$ in $\mathcal{I}%
(C)$ then $\psi_{C\ast v}(w)=\chi_{C}(w)$ and $\chi_{C\ast v}(w)=\psi_{C}(w)$;
and all other transitions of $F$ have the same $\phi,\chi,\psi$ labels with
respect to $C$ and $C\ast v$. The proposition now follows from Kotzig's
theorem, as every other Euler system of $F$ can be obtained from $C$ through
$\kappa$-transformations.
\end{proof}

%

%TCIMACRO{\FRAME{ftbFU}{4.5965in}{3.2621in}{0pt}{\Qcb{Transition labels with
%respect to $C$ and $C\ast v$.}}{\Qlb{tranmf14}}{tranmf14.ps}%
%width 4.5965in;  height 3.2621in;  depth 0pt;  original-width 8.4968in;
%original-height 11.0056in;  cropleft "0.1411";  croptop "0.8904";
%cropright "0.8576";  cropbottom "0.4986";
%filename 'tranmf14.ps';file-properties "XNPEU";}} }%
%BeginExpansion
\begin{figure}
[tb]
\begin{center}
\includegraphics[
trim=1.198899in 5.487392in 1.209945in 1.206214in,
height=3.2621in,
width=4.5965in
]%
{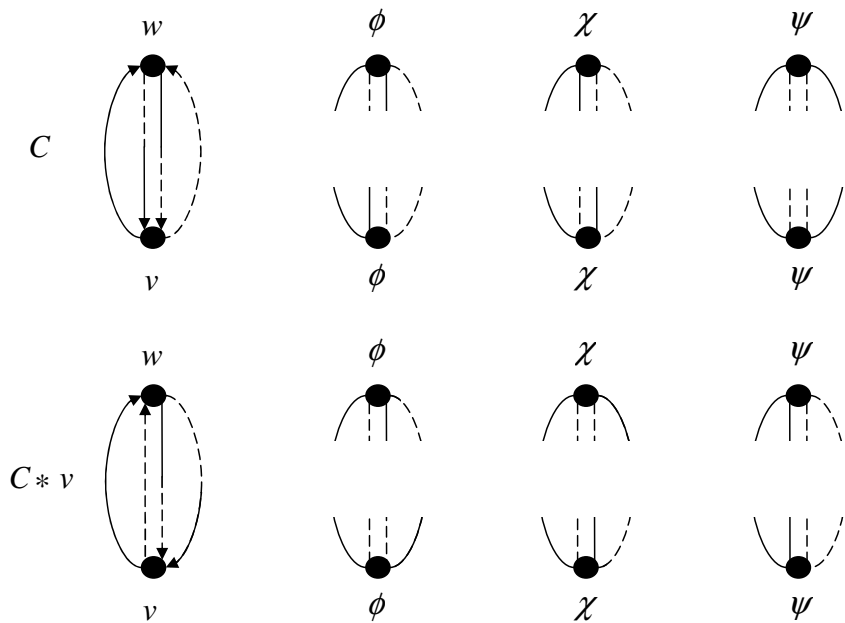}%
\caption{Transition labels with respect to $C$ and $C\ast v$.}%
\label{tranmf14}%
\end{center}
\end{figure}
%EndExpansion

Here is a simple analogy: a vector space has many bases, and each particular
basis gives rise to a coordinatization of the vector space. The
coordinatizations corresponding to different bases are not independent of each
other; once we know one coordinatization, all the others are determined
through multiplication by nonsingular matrices. This simple analogy is more
apt than it may appear at first glance, by the way; see \cite{T5, Tnew} for details.

\begin{definition}
\label{matrix}Let $F$ be a 4-regular graph with an Euler system $C$, let $I$
denote the identity matrix, and let $\mathcal{A}(\mathcal{I}(C))$ denote the
adjacency matrix of $\mathcal{I}(C)$. Then%
\[
M(C)=\left(  I\mid\mathcal{A}(\mathcal{I}(C))\mid I+\mathcal{A}(\mathcal{I}%
(C))\right)  \text{.}%
\]
The rows of this matrix are indexed using vertices of $F$ as in $\mathcal{A}%
(\mathcal{I}(C))$, and the columns are indexed using transitions of $F$ as
follows: the column of $I$ corresponding to $v$ is indexed by $\phi_{C}(v)$,
the column of $\mathcal{A}(\mathcal{I}(C))$ corresponding to $v$ is indexed by
$\chi_{C}(v)$, and the column of $I+\mathcal{A}(\mathcal{I}(C))$ corresponding
to $v$ is indexed by $\psi_{C}(v)$.
\end{definition}

Different Euler systems of $F$ will certainly give rise to different matrices,
but these matrices are tied together in a special way.

\begin{proposition}
\label{inv}If $C$ and $C^{\prime}$ are Euler systems of $F$ then $M(C)$ and
$M(C^{\prime})$ represent the same matroid on $\mathfrak{T}(F)$.
\end{proposition}

\begin{proof}
Clearly the matroid $M_{\tau}(F)$ is not affected if elementary row operations
are applied to $M(C)$; also, permuting the columns of $M(C)$ does not affect
the matroid so long as the same permutation is applied to the column indices.

According to Kotzig's theorem, it suffices to prove the proposition when
$C^{\prime}=C\ast v$. Each of $M(C)$, $M(C\ast v)$ consists of three square
submatrices, which we refer to as their $\phi$, $\chi$ and $\psi$\emph{
parts}. The columns of each part are indexed by $V(F)$. We can obtain $M(C\ast
v)$ from $M(C)$ using row and column operations as follows. First, add the $v$
row of $M(C)$ to every other row corresponding to a neighbor of $v$. Second,
interchange the $v$ columns of the $\phi$ and $\psi$ parts of the resulting
matrix. Third, for each neighbor $w$ of $v$, interchange the $w$ columns of
the $\chi$ and $\psi$ parts of the resulting matrix.

Notice that the column interchanges of the preceding paragraph correspond
precisely to the label changes induced by the $\kappa$-transformation
$C\mapsto C\ast v$, as discussed in the proof of Proposition \ref{labelinv}:
$\phi$ and $\psi$ are interchanged at $v$, and $\chi$ and $\psi$ are
interchanged at each neighbor $w$ of $v$. Consequently, $M(C)$ and $M(C\ast
v)$ represent the same matroid on $\mathfrak{T}(F)$.
\end{proof}

Proposition \ref{inv} allows us to make the following definition:

\begin{definition}
\label{matroid}If $F$ is a 4-regular graph with an Euler system $C$, then the
\emph{transition matroid} $M_{\tau}(F)$ is the binary matroid on
$\mathfrak{T}(F)$ represented by the matrix $M(C)$.
\end{definition}

It is not hard to verify that $M_{\tau}(F)$ satisfies the requirements of
Theorem \ref{tranmat}. That is, $M_{\tau}(F)$ is a rank-$n$ matroid defined on
$\mathfrak{T}(F)$ such that for each circuit partition $P$ of $F$, the rank of
$\tau(P)$ in $M_{\tau}(F)$ is given by $r(\tau(P))=n+c(F)-\left\vert
P\right\vert $. For if $C$ is any Euler system of $F$ then the $I$ submatrix
of $M(C)$ guarantees that the rank of $M(C)$ is $n$, no matter what
$\mathcal{A}(\mathcal{I}(C))$ is. Also, if $P$ is a circuit partition of $F$
then the rank of $\tau(P)$ in $M_{\tau}(F)$ is the $GF(2)$-rank of the
submatrix of $M(C)$ consisting of the columns corresponding to elements of
$\tau(P)$. As mentioned in the introduction, the fact that this rank equals
$n-\left\vert P\right\vert +c(F)$ has been discussed by many researchers. (The
reader who has not seen the circuit-nullity formula before may find the
account in \cite{Tnew} convenient, as the notation is close to ours.) This
completes the proof of Theorem \ref{tranmat}.

We briefly discuss some basic properties of transition matroids. Notice first
that the three columns of $M(C)$ corresponding to a vertex $v$ sum to $0$, and
the $\phi_{C}(v)$ and $\psi_{C}(v)$ columns are never 0. We conclude that the
vertex triple $\{\phi_{C}(v)$, $\chi_{C}(v)$, $\psi_{C}(v)\}$ contains either
a 3-cycle of $M_{\tau}(F)$, or a loop and a pair of non-loop parallel
elements. Classes of the latter sort correspond to isolated vertices in
interlacement graphs of $F$, i.e., cut vertices and looped vertices of $F$.

As the rank of $M_{\tau}(F)$ is $n$, the circuit-nullity formula
$r(\tau(P))=n-\left\vert P\right\vert +c(F)$ indicates that a circuit
partition $P$ is an Euler system of $F$ if and only if $\tau(P)$ is a basis of
$M_{\tau}(F)$. Here are two important features of the relationship between
Euler systems and matroid bases.

1. A fundamental property of matroids is this: the bases of a matroid are all
interconnected by basis exchanges (single-element swaps between bases). This
property holds in $M_{\tau}(F)$, of course, but it does not mean that Euler
systems are interconnected by single-vertex transition changes. In general
$M_{\tau}(F)$ will have many bases that do not correspond to Euler systems
because they do not include precisely one transition for each vertex. For
comparison we should regard $\kappa$-transformations as specially modified
basis exchanges, and we should regard Kotzig's theorem as asserting that those
bases of $M_{\tau}(F)$ that correspond to Euler systems are distributed
densely enough that they form a connected set under these special exchanges.

2. A basis of a matroid is a subset of the matroid's ground set, so it is a
simple matter to compare two bases $B$ and $B^{\prime}$; when they disagree
about an element it must be that the element is included in one of
$B,B^{\prime}$ and excluded from the other. For $M_{\tau}(F)$ this simple
analysis is correct if we focus our attention on $\mathfrak{T}(F)$, but the
fact that there are three different transitions at each vertex complicates
matters if we focus our attention on $V(F)$ instead of $\mathfrak{T}(F)$.
There are five different ways that two Euler systems $C,C^{\prime}$ might
disagree about a vertex $v$, corresponding to the five nontrivial permutations
of the set $\{\phi,\chi,\psi\}$.%

%TCIMACRO{\FRAME{ftbFU}{2.4768in}{0.8302in}{0pt}{\Qcb{Detachment along a
%transition.}}{\Qlb{tranmf4}}{tranmf4.ps}%
%width 2.4768in;  height 0.8302in;  depth 0pt;  original-width 8.4968in;
%original-height 11.0056in;  cropleft "0.3305";  croptop "0.8287";
%cropright "0.7147";  cropbottom "0.7313";
%filename 'tranmf4.ps';file-properties "XNPEU";}} }%
%BeginExpansion
\begin{figure}
[tb]
\begin{center}
\includegraphics[
trim=2.808192in 8.048395in 2.424137in 1.885259in,
height=0.8302in,
width=2.4768in
]%
{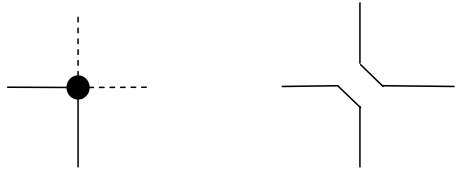}%
\caption{Detachment along a transition.}%
\label{tranmf4}%
\end{center}
\end{figure}
%EndExpansion

A transition matroid is a $3n$-element matroid of rank $n$, so deleting or
contracting a single element cannot yield another transition matroid. However,
if the triple corresponding to a vertex $v$ is removed by contracting one
transition at $v$ and deleting the other two, then the resulting matroid is
isomorphic to the transition matroid of the 4-regular graph $F^{\prime}$
obtained from $F$ by detachment (or splitting)\ along the contracted
transition. (As indicated in Fig. \ref{tranmf4}, the detachment is constructed
by first removing $v$, and then connecting the loose half-edges paired by the
contracted transition. If $v$ is looped in $F$, this construction may produce
one or two arcs with no incident vertex; any such arc is simply discarded.)
This property is easy to verify if we first choose an Euler system $C$ in
which the contracted transition is $\phi_{C}(v)$. Then the only nonzero entry
of the column of $M(C)$ corresponding to the contracted transition is a $1$ in
the $v$ row, so the definition of matroid contraction tells us that the
matroid $(M_{\tau}(F)/\phi_{C}(v))-\chi_{C}(v)-\psi_{C}(v)$ is represented by
the submatrix of $M(C)$ obtained by removing the row corresponding to $v$ and
removing all three columns corresponding to $v$. The resulting matrix is
clearly the same as $M(C^{\prime})$, where $C^{\prime}$ is the Euler system of
$F^{\prime}$ obtained in the natural way from $C$.

Many binary matroids occur as submatroids of transition matroids of 4-regular
graphs. Jaeger \cite{J3} proved that every cographic matroid is represented by
a symmetric matrix whose off-diagonal entries equal those of the interlacement
matrix of some 4-regular graph; consequently, the submatroids of transition
matroids include all cographic matroids. (See the end of Section 7 for more
discussion of this point.) In addition, many non-cographic matroids occur as
submatroids of transition matroids. For example, consider the 4-regular graph
$F$ of Fig. \ref{tranmf12}. Every pair of vertices is interlaced with respect
to the indicated Euler circuit, so $M_{\tau}(F)$ is represented by the matrix
\[%
\begin{pmatrix}
1 & 0 & 0 & 0 & 1 & 1 & 1 & 1 & 1\\
0 & 1 & 0 & 1 & 0 & 1 & 1 & 1 & 1\\
0 & 0 & 1 & 1 & 1 & 0 & 1 & 1 & 1
\end{pmatrix}
.
\]
The first seven columns represent a well-known non-cographic submatroid of
$M_{\tau}(F)$, the Fano matroid.%

%TCIMACRO{\FRAME{ftbpFU}{1.7798in}{0.8527in}{0pt}{\Qcb{The indicated Euler
%circuit has $\QTR{cal}{I}(C)\cong K_{3}$.}}{\Qlb{tranmf12}}{tranmf12.ps}%
%width 1.7798in;  height 0.8527in;  depth 0pt;  original-width 8.4968in;
%original-height 11.0056in;  cropleft "0.4092";  croptop "0.8904";
%cropright "0.6839";  cropbottom "0.7904";
%filename 'tranmf12.ps';file-properties "XNPEU";}} }%
%BeginExpansion
\begin{figure}
[ptb]
\begin{center}
\includegraphics[
trim=3.476891in 8.698827in 2.685839in 1.206214in,
height=0.8527in,
width=1.7798in
]%
{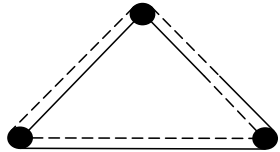}%
\caption{The indicated Euler circuit has $\mathcal{I}(C)\cong K_{3}$.}%
\label{tranmf12}%
\end{center}
\end{figure}
%EndExpansion

Before proceeding, we should mention that it is likely that some form of the
transition matroid has been noticed before. See \cite{B1}, where Bouchet
comments that \textquotedblleft the main applications -- for instance, to
eulerian graphs -- involve multimatroids that can be sheltered by
matroids.\textquotedblright\ (In Bouchet's terminology, the rank function of
the multimatroid associated with a 4-regular graph $F$ is the restriction of
the rank function of $M_{\tau}(F)$ to the subtransversals of the vertex
triples of $\mathfrak{T}(F)$; he defined this restriction using the
circuit-nullity formula. $M_{\tau}(F)$ shelters\ the multimatroid because its
rank function extends the multimatroid's rank function to arbitrary subsets of
$\mathfrak{T}(F)$.) As far as we know, though, Bouchet's published work
includes neither any explicit description of his sheltering\ matroid nor any
results about it.

The reason we mention \textquotedblleft some form\ of the transition
matroid\textquotedblright\ is that there are at least three other ways to
define $M_{\tau}(F)$.

1. The rank function of $M_{\tau}(F)$ may be described in two stages: First,
use the circuit-nullity formula to determine $r(\tau(P))$ as
$n+c(F)-\left\vert P\right\vert $\ for each circuit partition $P$.\ Second,
describe how the ranks of subsets of the form $\tau(P)$ determine the ranks of
all other subsets of $\mathfrak{T}(F)$. Bouchet detailed the first stage in
\cite{B1}, but did not discuss the second stage. It is possible to complete
this description without explicitly mentioning the matrix $M(C)$; the crucial
fact is that each vertex triple has rank one or two. We leave the details as
an exercise for the interested reader.

2. Similarly, the bases of the matroid $M_{\tau}(F)$ may be described in two
stages: First, observe that $\tau(C)$ is a basis for each Euler system $C$.
Second, describe the rest of the bases. Bouchet noted the connection between
Euler systems and submatrices of $\mathcal{A}(\mathcal{I}(C))$ in \cite{Bdr},
but did not discuss the other bases of $M_{\tau}(F)$.

3. The line graph $L(F)$ is a 6-regular graph that has a vertex for each edge
of $F$ and an edge for each pair of distinct half-edges incident at the same
vertex. (With this convention, a loop in $F$ gives rise to a loop in $L(F)$.)
That is, there is an edge of $L(F)$ for each single transition of $F$. Like
any graph, $L(F)$ has a cycle matroid $M(L(F))$ consisting of vectors in the
$GF(2)$-vector space $GF(2)^{V(L(F))}$; the vector corresponding to a non-loop
edge of $L(F)$ has nonzero coordinates corresponding to the two vertices
incident on that edge, and the vector corresponding to a loop is $0$. If $Z$
is the subspace of $GF(2)^{V(L(F))}$ spanned by $\{z_{1}-z_{2}\mid$ $z_{1}$
and $z_{2}$ are the vectors corresponding to two single transitions that
constitute a transition of $F\}$ then $M_{\tau}(F)$ is the matroid on
$\mathfrak{T}(F)$ in which each transition is represented by the image of
either of its constituent single transitions in the quotient space
$GF(2)^{V(L(F))}/Z$.

\section{Proof of Theorem \ref{adjiso}}

In this section we provide a brief outline of arguments verifying equivalences
among the conditions of Theorem \ref{adjiso}. The substance of the arguments
appears in the references.

Recall that Theorem \ref{adjiso} asserts equivalences among the following
statements about 4-regular graphs $F$ and $F^{\prime}$:

\begin{enumerate}
\item The transition matroid of $F$ is isomorphic to the transition matroid of
$F^{\prime}$.

\item The isotropic system of $F$ is isomorphic to the isotropic system of
$F^{\prime}$.

\item Any $\Delta$-matroid of $F$ is isomorphic to a $\Delta$-matroid of
$F^{\prime}$.

\item All $\Delta$-matroids of $F$ are isomorphic to $\Delta$-matroids of
$F^{\prime}$.

\item Any interlacement graph of $F$ is isomorphic to an interlacement graph
of $F^{\prime}$.

\item All interlacement graphs of $F$ are isomorphic to interlacement graphs
of $F^{\prime}$.
\end{enumerate}

The implication 6 $\Rightarrow$ 5 is obvious, and the implication 5
$\Rightarrow$ 1 follows immediately from Definition \ref{matroid}. The
implication 1 $\Rightarrow$ 6 follows from two results. The first of these
results is a theorem of the author \cite{Tnewnew}, which tells us that if $G$
and $G^{\prime}$ are simple graphs with adjacency matrices $A$ and $A^{\prime
}$, then the matrices%
\[
\left(  I\mid A\mid I+A\right)  \text{ and }\left(  I\mid A^{\prime}\mid
I+A^{\prime}\right)
\]
represent isomorphic binary matroids if and only if $G$ and $G^{\prime}$ are
equivalent under simple local complementation. The second of these results is
the fundamental relationship between $\kappa$-transformations and local
complementation, which tells us that any sequence of simple local
complementations on interlacement graphs is induced by a corresponding
sequence of $\kappa$-transformations on Euler systems. This implies that if
$F$ and $F^{\prime}$ are 4-regular graphs with Euler systems $C$ and
$C^{\prime}$, then $\mathcal{I}(C)$ and $\mathcal{I}(C^{\prime})$ are
equivalent under simple local complementation if and only if $F^{\prime}$ has
an Euler system $C^{\prime\prime}$ such that $\mathcal{I}(C)$ and
$\mathcal{I}(C^{\prime\prime})$ are isomorphic.

The equivalence of statements 2, 5 and 6 of Theorem \ref{adjiso} is readily
extracted from Bouchet's discussion of graphic isotropic systems in Section 6
of \cite{Bi2}, although he does not give quite the same statements. Similarly,
the equivalence of statements 3, 4, 5 and 6 may be extracted from Section 5 of
\cite{Bdr}. We do not repeat the details of these discussions, but we might
summarize them by saying that as the isotropic system of $F$ is determined by
the sizes of the circuit partitions of $F$, and the $\Delta$-matroids
associated with $F$ are determined by the Euler systems of $F$, these
structures are tied to the transition matroid $M_{\tau}(F)$ by the
circuit-nullity formula. A more thorough account of the connections tying
isotropic matroids to $\Delta$-matroids and isotropic systems is given in
\cite{Tnewnew}.

In the balance of the paper we focus on transition matroids and interlacement graphs.

\section{Ghier's theorem}

Ghier \cite{gh} characterized the double occurrence words with connected,
isomorphic interlacement graphs, and several years later a similar
characterization of chord diagrams with isomorphic interlacement graphs was
independently discovered by Chmutov and Lando \cite{ChL}. In this section we
deduce Theorem \ref{iso} from Ghier's theorem.

\subsection{Ghier's theorem for double occurrence words}

A double occurrence word in letters $v_{1},...,v_{n}$ is a word of length $2n$
in which every $v_{i}$ occurs twice, and the interlacement graph of a double
occurrence word is the simple graph with vertices $v_{1},...,v_{n}$ in which
$v_{i}$ and $v_{j}$ are adjacent if and only if they are interlaced, i.e.,
they appear in the order $v_{i}...v_{j}...v_{i}...v_{j}$ or $v_{j}%
...v_{i}...v_{j}...v_{i}$.

\begin{definition}
An \emph{equivalence} between two double occurrence words is a finite sequence
of cyclic permutations $x_{1}...x_{2n}\mapsto x_{i}...x_{2n}x_{1}...x_{i-1}$
and reversals $x_{1}...x_{2n}\mapsto x_{2n}...x_{1}$.
\end{definition}

\begin{definition}
\cite{gh} Let $\{v_{1},...,v_{n}\}$ be the union of two disjoint subsets,
$V^{\prime}$ and $V^{\prime\prime}$. Let $W$ be a double occurrence word in
$v_{1},...,v_{n}$ of the form $W_{1}^{\prime}W_{1}^{\prime\prime}W_{2}%
^{\prime}W_{2}^{\prime\prime}$, where $W_{1}^{\prime}$, $W_{2}^{\prime}$
involve only letters in $V^{\prime}$ and $W_{1}^{\prime\prime}$,
$W_{2}^{\prime\prime}$ involve only letters in $V^{\prime\prime}$. Then the
corresponding \emph{turnarounds} of $W$ are the double occurrence words
$W_{1}^{\prime}\overleftarrow{W}_{1}^{\prime\prime}W_{2}^{\prime
}\overleftarrow{W}_{2}^{\prime\prime}$, $\overleftarrow{W}_{1}^{\prime}%
W_{1}^{\prime\prime}\overleftarrow{W}_{2}^{\prime}W_{2}^{\prime\prime}$ and
$\overleftarrow{W}_{1}^{\prime}\overleftarrow{W}_{1}^{\prime\prime
}\overleftarrow{W}_{2}^{\prime}\overleftarrow{W}_{2}^{\prime\prime}$, where
the arrow indicates reversal of a subword.
\end{definition}

In Ghier's definition each of $V^{\prime}$, $V^{\prime\prime}$ must have at
least two elements, but there is no harm in ignoring this stipulation because
the effect is simply to include some equivalences among the turnarounds. Ghier
observed that the first type of turnaround suffices, up to equivalence: a
turnaround of the second type is a composition of cyclic permutations with a
turnaround of the first type, and a turnaround of the third type is a
composition of turnarounds of the first and second types. There are also
several other types of operations that can be obtained by composing
turnarounds with equivalences. Some of these operations involve no actual
\textquotedblleft turning around.\textquotedblright\ For instance,
$W_{1}^{\prime}W_{1}^{\prime\prime}W_{2}^{\prime}W_{2}^{\prime\prime}\mapsto$
$W_{1}^{\prime}W_{2}^{\prime\prime}W_{2}^{\prime}W_{1}^{\prime\prime}$ is
obtained by applying a reversal $W_{1}^{\prime}W_{1}^{\prime\prime}%
W_{2}^{\prime}W_{2}^{\prime\prime}\mapsto$ $\overleftarrow{W}_{2}%
^{\prime\prime}\overleftarrow{W}_{2}^{\prime}\overleftarrow{W}_{1}%
^{\prime\prime}\overleftarrow{W}_{1}^{\prime}$, a turnaround $\overleftarrow
{W}_{2}^{\prime\prime}\overleftarrow{W}_{2}^{\prime}\overleftarrow{W}%
_{1}^{\prime\prime}\overleftarrow{W}_{1}^{\prime}\mapsto$ $W_{2}^{\prime
\prime}W_{2}^{\prime}W_{1}^{\prime\prime}W_{1}^{\prime}$, and then a cyclic
permutation $W_{2}^{\prime\prime}W_{2}^{\prime}W_{1}^{\prime\prime}%
W_{1}^{\prime}\mapsto$ $W_{1}^{\prime}W_{2}^{\prime\prime}W_{2}^{\prime}%
W_{1}^{\prime\prime}$.

It is not difficult to see that equivalences and turnarounds preserve
interlacement graphs. Building on earlier work of Bouchet \cite{Bec}, Ghier
proved the following converse:

\begin{theorem}
\cite{gh} Two double occurrence words have isomorphic, connected interlacement
graphs if and only if some finite sequence of equivalences and turnarounds
transforms one word into the other.
\end{theorem}

In order to consider double occurrence words with disconnected interlacement
graphs, let us define a \emph{concatenation }of two disjoint double occurrence
words $w_{1}...w_{2m}$ and $x_{1}...x_{2n}\,$\ in the obvious way: it is
$w_{1}...w_{2m}x_{1}...x_{2n}$ or $x_{1}...x_{2n}w_{1}...w_{2m}$. (Here
\textquotedblleft disjoint\textquotedblright\ means that no letter occurs in
both words.) Clearly then the interlacement graph of a concatenation is the
disjoint union of the interlacement graphs of the concatenated words. The
converse holds too:

\begin{lemma}
\label{disconnect}The interlacement graph of a double occurrence word is
disconnected if and only if the word is equivalent to a concatenation of two
nonempty subwords.
\end{lemma}

\begin{proof}
Suppose $x_{1}...x_{2n}$ is a double occurrence word in $v_{1},...,v_{n}$,
with a disconnected interlacement graph. Using a cyclic permutation if
necessary, we may presume that $x_{1}$ and $x_{2n}$ lie in different connected
components of the interlacement graph. Let $i$ be the greatest index such that
$x_{i}$ lies in the same connected component of the interlacement graph as
$x_{1}$; necessarily then $x_{i}$ is the second appearance of the
corresponding letter.

We claim that there is no $j>i$ such that $x_{j}=x_{j^{\prime}}$ with
$j^{\prime}<i$. Suppose instead that there is such a $j$. As $j>i$, $x_{j}$ is
not in the connected component of the interlacement graph that contains
$x_{1}$ and $x_{i}$, so $x_{j}$ is not interlaced with either; hence the
double occurrence word must be of the form $x_{1}Ax_{1}Bx_{j}Cx_{i}%
Dx_{i}Ex_{j}F$ for some double occurrence word $ABCDEF$ in the remaining
letters. Choose a path from $x_{1}$ to $x_{i}$ in the interlacement graph; say
it is $x_{1},x_{k_{1}},x_{k_{2}},...,x_{k_{p-1}},x_{i}$. As $x_{k_{1}}$ is
interlaced with $x_{1}$, $x_{k_{1}}$ must appear precisely once in $A$. As
$x_{k_{1}}$ is not interlaced with $x_{j}$, its other appearance must be in
$B$ or $F$; hence $x_{k_{1}}$ appears twice in $Fx_{1}Ax_{1}B$. As $x_{k_{2}}$
is interlaced with $x_{k_{1}}$, $x_{k_{2}}$ must appear at least once in
$Fx_{1}Ax_{1}B$. As $x_{k_{2}}$ is not interlaced with $x_{j}$, its other
appearance cannot be in $Cx_{i}Dx_{i}E$, so it must be in $Fx_{1}Ax_{1}B$;
hence $x_{k_{2}}$ appears twice in $Fx_{1}Ax_{1}B$. Similarly, $x_{k_{3}}$
must appear at least once in $Fx_{1}Ax_{1}B$ (as it is interlaced with
$x_{k_{2}}$) and cannot appear precisely once in $Fx_{1}Ax_{1}B$ (as it is not
interlaced with $x_{j}$), so $x_{k_{3}}$ also appears twice in $Fx_{1}Ax_{1}%
B$. The same reasoning holds for $x_{k_{4}}$, $x_{k_{5}}$, and so on.
Eventually we conclude that $x_{k_{p-1}}$ appears twice in $Fx_{1}Ax_{1}B$;
but this contradicts the fact that $x_{k_{p-1}}$ is interlaced with $x_{i}$.

The contradiction verifies the claim. It follows that $x_{1}...x_{2n}$ is the
concatenation of $x_{1}...x_{i}$ and $x_{i+1}...x_{2n}$.
\end{proof}

\begin{definition}
A \emph{disjoint family} of double occurrence words is a finite set $\{W_{1}$,
..., $W_{k}\}$ of double occurrence words, no two of which share any letter.
The \emph{interlacement graph} of $\{W_{1},...,W_{k}\}$ is the union of the
interlacement graphs of $W_{1},...,W_{k}$.
\end{definition}

In particular, if $W$ is a double occurrence word then $\{W\}$ is a singleton
disjoint family.

\begin{definition}
\label{equivalence}An \emph{equivalence} between two disjoint families of
double occurrence words is some sequence of the following operations:
cyclically permuting a word, reversing a word, replacing two disjoint words
with their concatenation, and separating a concatenation of two disjoint words
into the two separate words.
\end{definition}

If $\{W_{1},...,W_{k}\}$ is a disjoint family of double occurrence words whose
interlacement graph has $c$ connected components then $1\leq k\leq c$, and
$k=c$ if and only if each $W_{i}$ has a connected interlacement graph.
Considering Lemma \ref{disconnect}, we see that $\{W_{1},...,W_{k}\}$ is
equivalent to disjoint families of all cardinalities in the set $\{1,...,c\}$.

\begin{definition}
\label{turnaround}A \emph{turnaround} of a disjoint family of double
occurrence words is a disjoint family obtained by applying a turnaround to
some element of the original disjoint family.
\end{definition}

Here is the general form of Ghier's theorem.

\begin{corollary}
\label{ghier}Two disjoint families of double occurrence words have isomorphic
interlacement graphs if and only if one disjoint family can be obtained from
the other by some finite sequence of equivalences and turnarounds.
\end{corollary}

\begin{proof}
If necessary, replace each disjoint family by an equivalent one in which every
double occurrence word has a connected interlacement graph. Then apply the
original form of Ghier's theorem to the individual words.
\end{proof}

\subsection{Ghier's theorem for 4-regular graphs}

An Euler system $C$ in a 4-regular graph $F$ with $c(F)$ connected components
gives rise to many equivalent disjoint families of double occurrence words
$W(C)=\{W_{1}(C),...,W_{c(F)}(C)\}$. For $1\leq i\leq c(F)$ the double
occurrence word $W_{i}(C)$ is obtained by reading off the vertices visited by
the $i$th element of $C$, in order. Clearly $W(C)$ is defined only up to
cyclic permutations and reversals of its elements, but in any case the
interlacement graph of $W(C)$ is $\mathcal{I}(C)$. Corollary \ref{ghier}
implies the following.

\begin{corollary}
\label{ghi} Let $F$ and $F^{\prime}$ be 4-regular graphs with Euler systems
$C$ and $C^{\prime}$. Then $\mathcal{I}(C)\cong\mathcal{I}(C^{\prime})$ if and
only if $W(C)$ and $W(C^{\prime})$ differ only by equivalences and turnarounds.
\end{corollary}

On the other hand, a disjoint family $W=\{W_{1},...,W_{k}\}$ of double
occurrence words in the letters $v_{1},...,v_{n}$ gives rise to an Euler
system $C(W)$ in a 4-regular graph $F(W)$ with $V(F(W))=\{v_{1},...,v_{n}\}$.
$F(W)$ has $k$ connected components, and the vertices of the $i$th connected
component are the letters $v_{j}$ that appear in $W_{i}$. If $W_{i}%
=x_{1}...x_{2p}$ then for $1\leq q<2p$ there is an edge of this connected
component from $x_{q}$ to $x_{q+1}$, and there is also an edge from $x_{2p}$
to $x_{1}$. $C(W)$ includes the Euler circuits obtained directly from
$W_{1},...,W_{k}$. Clearly $W$ is equivalent to $W(C(W))$, and $C(W(C))=C$ for
any\ Euler system $C$ in a 4-regular graph.%

%TCIMACRO{\FRAME{ftbpFU}{4.0897in}{2.1136in}{0pt}{\Qcb{Orientation-reversing
%and orientation-preserving balanced mutations.}}{\Qlb{tranmf5b}}%
%{tranmf5b.ps}{\special{ language "Scientific Word";  type "GRAPHIC";
%width 4.0897in;  height 2.1136in;  depth 0pt;  original-width 8.4968in;
%original-height 11.0056in;  cropleft "0.2202";  croptop "0.8844";
%cropright "0.8571";  cropbottom "0.6318";
%filename 'tranmf5b.ps';file-properties "XNPEU";}} }%
%BeginExpansion
\begin{figure}
[ptb]
\begin{center}
\includegraphics[
trim=1.870996in 6.953338in 1.214193in 1.272247in,
height=2.1136in,
width=4.0897in
]%
{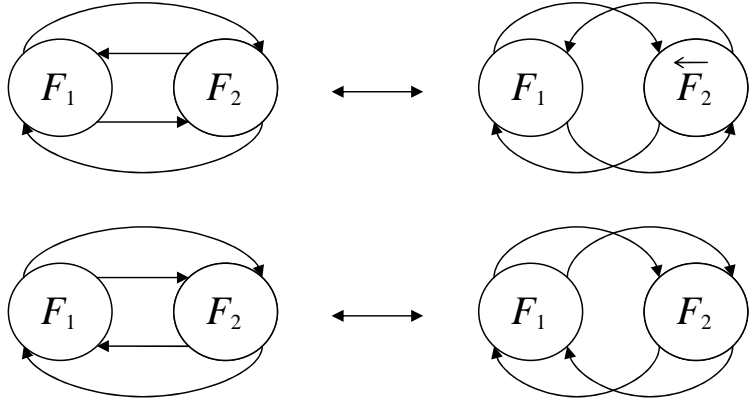}%
\caption{Orientation-reversing and orientation-preserving balanced mutations.}%
\label{tranmf5b}%
\end{center}
\end{figure}
%EndExpansion

A double occurrence word is oriented in a natural way: we read the word from
beginning to end. In order to deduce Theorem \ref{iso} of the introduction
from Corollary \ref{ghi}, then, we must understand\ the balanced mutations of
Definition \ref{mutation} as operations on directed graphs. Recall that a
direction of an edge of a graph is determined by designating one of the
half-edges as initial, and the other as terminal.\ A balanced orientation of a
4-regular graph is a set of edge-directions with the property that at each
vertex, there are two initial half-edges and two terminal half-edges. It is
easy to see that every 4-regular graph $F$ has balanced orientations: just
walk along the circuits of an Euler system to determine edge-directions. If
$F$ has subgraphs $F_{1}$ and $F_{2}$ as in Definition \ref{mutation}, then a
balanced orientation of $F$ must include two edges directed from $F_{1}$ to
$F_{2}$, and two edges directed from $F_{2}$ to $F_{1}$. There are three
different pairings of these four edges; two of them pair oppositely directed
edges, and the third pairs like-directed edges. No matter which edge-pairing
is used, balanced orientations of $F$ and $F^{\prime}$ correspond to each
other directly. This correspondence is illustrated in Fig. \ref{tranmf5b},
where the notation $\overleftarrow{F}_{2}$ is used to indicate that
edge-directions within $F_{2}$ have been reversed.

Fig. \ref{tranmf5b} depicts two types of balanced mutations but as indicated
in Fig. \ref{tranmf5c}, the first type actually suffices: a balanced mutation
involving pairs of like-directed edges can be performed by combining two
balanced mutations involving pairs of oppositely-directed edges. In Fig.
\ref{tranmf5c} the edge-pairing used in the second of these
orientation-reversing mutations is indicated by dashes.%

%TCIMACRO{\FRAME{ftbpFU}{4.8447in}{0.9504in}{0pt}{\Qcb{Orientation-reversing
%balanced mutations suffice.}}{\Qlb{tranmf5c}}{tranmf5c.ps}%
%width 4.8447in;  height 0.9504in;  depth 0pt;  original-width 8.4968in;
%original-height 11.0056in;  cropleft "0.1492";  croptop "0.8905";
%cropright "0.9052";  cropbottom "0.7787";
%filename 'tranmf5c.ps';file-properties "XNPEU";}} }%
%BeginExpansion
\begin{figure}
[ptb]
\begin{center}
\includegraphics[
trim=1.267723in 8.570061in 0.805497in 1.205113in,
height=0.9504in,
width=4.8447in
]%
{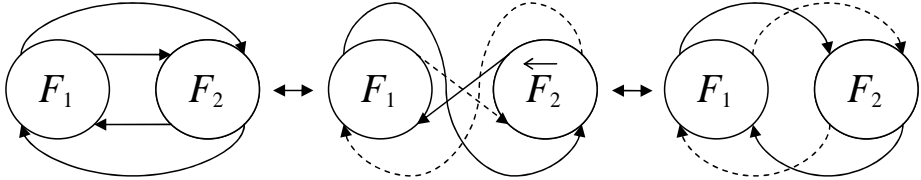}%
\caption{Orientation-reversing balanced mutations suffice.}%
\label{tranmf5c}%
\end{center}
\end{figure}
%EndExpansion

\begin{theorem}
\label{ghier2} Let $F$ and $F^{\prime}$ be 4-regular graphs. Then $F$ and
$F^{\prime}$ have Euler systems $C$ and $C^{\prime}$ with $\mathcal{I}%
(C)\cong\mathcal{I}(C^{\prime})$ if and only if $F$ and $F^{\prime}$ differ
only by connected sums, separations and balanced mutations.
\end{theorem}

\begin{proof}
If $F$ and $F^{\prime}$ have Euler systems $C$ and $C^{\prime}$ with
$\mathcal{I}(C)\cong\mathcal{I}(C^{\prime})$ then Corollary \ref{ghi} tells us
that $W(C^{\prime})$ can be obtained from $W(C)$ by equivalences and
turnarounds. Applying cyclic permutations and reversals to elements of $W$ has
no effect on $F(W)$ or $C(W)$; a concatenation has the effect of replacing two
connected components of $F(W)$ with a connected sum; separating a
concatenation has the effect of separating a connected sum; and a turnaround
of $W$ yields a balanced mutation of $F(W)$. See Fig. \ref{tranmf20} for a
schematic illustration of the turnaround $W_{1}^{\prime}W_{1}^{\prime\prime
}W_{2}^{\prime}W_{2}^{\prime\prime}\mapsto W_{1}^{\prime}\overleftarrow{W}%
_{1}^{\prime\prime}W_{2}^{\prime}\overleftarrow{W}_{2}^{\prime\prime}$. (N.b.
The walks indicated by dashed line segments within $F_{1}$ and $F_{2}$ are
edge-disjoint but not necessarily vertex-disjoint.)%
%TCIMACRO{\FRAME{ftbpFU}{4.0897in}{1.0006in}{0pt}{\Qcb{A turnaround of $W$
%yields a balanced mutation of $F(W)$.}}{\Qlb{tranmf20}}{tranmf20.ps}%
%width 4.0897in;  height 1.0006in;  depth 0pt;  original-width 8.4968in;
%original-height 11.0056in;  cropleft "0.2203";  croptop "0.8903";
%cropright "0.8575";  cropbottom "0.7724";
%filename 'tranmf20.ps';file-properties "XNPEU";}} }%
%BeginExpansion
\begin{figure}
[ptb]
\begin{center}
\includegraphics[
trim=1.871845in 8.500726in 1.210794in 1.207315in,
height=1.0006in,
width=4.0897in
]%
{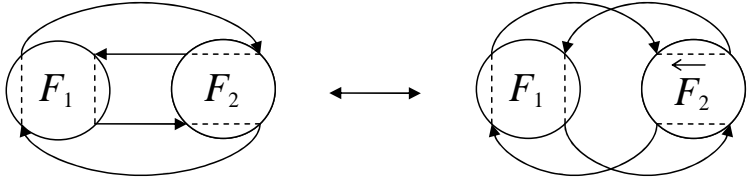}%
\caption{A turnaround of $W$ yields a balanced mutation of $F(W)$.}%
\label{tranmf20}%
\end{center}
\end{figure}
%EndExpansion

For the converse, suppose $F^{\prime}$ can be obtained from $F$ by a single
connected sum, separation or balanced mutation. Let $C$ be any Euler system of
$F$. If $F^{\prime}$ is obtained from $F$ with a single connected sum
operation then clearly $F^{\prime}$ has an Euler system $C^{\prime}$ such that
$W(C^{\prime})$ is obtained from $W(C)$ by concatenating two words. If
$F^{\prime}$ is obtained from $F$ by separating a connected sum then clearly
$F^{\prime}$ has an Euler system $C^{\prime}$ such that $W(C^{\prime})$ is
obtained from $W(C)$ by first (if necessary) cyclically permuting one word to
yield a concatenation of two disjoint subwords, and then replacing the
concatenation with the two disjoint words.

Suppose now that $F^{\prime}$ is obtained from $F$ by a balanced mutation. As
noted above, we may presume that $F$ and $F^{\prime}$ are given with balanced
orientations that agree within $F_{1}$ and disagree within $F_{2}$. Let $C$ be
a directed Euler system of $F$. Then the four edges connecting $F_{1}$ to
$F_{2}$ appear in either two circuits or one circuit of $C$, according to
whether or not they lie in different connected components of $F$. Fig.
\ref{tranmf18} indicates schematically that in either case, we can construct a
directed Euler system $C^{\prime}$ of $F^{\prime}$ by using the new edges to
connect the walks within $F_{1}$ and $F_{2}$ that constitute Euler circuits in
$C$. (The Euler circuits of other connected components are the same in
$C^{\prime}$ and $C$.) In the first case, $W(C^{\prime})$ can be obtained from
$W(C)$ using concatenations and separations. In the second case, $W(C^{\prime
})$ is obtained from $W(C)$ by a transformation of the form $W_{1}W_{2}%
X_{1}X_{2}\mapsto W_{1}\overleftarrow{X_{2}}X_{1}\overleftarrow{W_{2}}$, with
$W_{1}$ and $X_{1}$ the indicated walks within $F_{1}$. This transformation is
the composition of a reversal $W_{1}W_{2}X_{1}X_{2}\mapsto\overleftarrow
{X_{2}}\overleftarrow{X_{1}}\overleftarrow{W_{2}}\overleftarrow{W_{1}}$, a
turnaround $\overleftarrow{X_{2}}\overleftarrow{X_{1}}\overleftarrow{W_{2}%
}\overleftarrow{W_{1}}\mapsto\overleftarrow{X_{2}}X_{1}\overleftarrow{W_{2}%
}W_{1}$, and a cyclic permutation $\overleftarrow{X_{2}}X_{1}\overleftarrow
{W_{2}}W_{1}\mapsto W_{1}\overleftarrow{X_{2}}X_{1}\overleftarrow{W_{2}}$.%

%TCIMACRO{\FRAME{ftbpFU}{4.0923in}{2.0565in}{0pt}{\Qcb{Directed Euler systems
%correspond under balanced mutation.}}{\Qlb{tranmf18}}{tranmf18.ps}%
%width 4.0923in;  height 2.0565in;  depth 0pt;  original-width 8.4968in;
%original-height 11.0056in;  cropleft "0.2197";  croptop "0.8895";
%cropright "0.8571";  cropbottom "0.6440";
%filename 'tranmf18.ps';file-properties "XNPEU";}} }%
%BeginExpansion
\begin{figure}
[ptb]
\begin{center}
\includegraphics[
trim=1.866747in 7.087606in 1.214193in 1.216119in,
height=2.0565in,
width=4.0923in
]%
{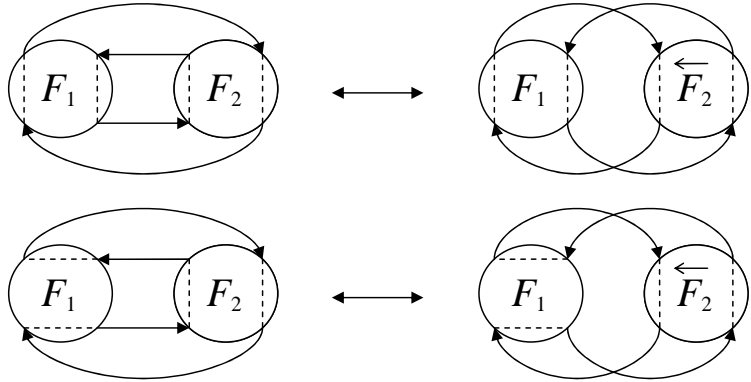}%
\caption{Directed Euler systems correspond under balanced mutation.}%
\label{tranmf18}%
\end{center}
\end{figure}
%EndExpansion

We conclude that if $F^{\prime}$ can be obtained from $F$ by connected sums,
separations and balanced mutations then it has an Euler system $C^{\prime}$
such that $W(C^{\prime})$ is obtained from $W(C)$ by equivalences and
turnarounds. Corollary \ref{ghi} then implies that $\mathcal{I}(C)\cong%
\mathcal{I}(C^{\prime})$.
\end{proof}

Theorem \ref{ghier2} allows us to deduce Theorem \ref{iso} from Theorem
\ref{adjiso}.%
%TCIMACRO{\FRAME{ftbpFU}{4.1866in}{2.0998in}{0pt}{\Qcb{Euler systems need not
%correspond under \textquotedblleft unbalanced mutation.\textquotedblright}%
%}{\Qlb{tranmf19}}{tranmf19.ps}{\special{ language "Scientific Word";
%type "GRAPHIC";  maintain-aspect-ratio TRUE;  display "USEDEF";
%valid_file "F";  width 4.1866in;  height 2.0998in;  depth 0pt;
%original-width 8.4968in;  original-height 11.0056in;  cropleft "0.2362";
%croptop "0.8898";  cropright "0.8888";  cropbottom "0.6386";
%filename 'tranmf19.ps';file-properties "XNPEU";}} }%
%BeginExpansion
\begin{figure}
[ptb]
\begin{center}
\includegraphics[
trim=2.006944in 7.028176in 0.944844in 1.212817in,
height=2.0998in,
width=4.1866in
]%
{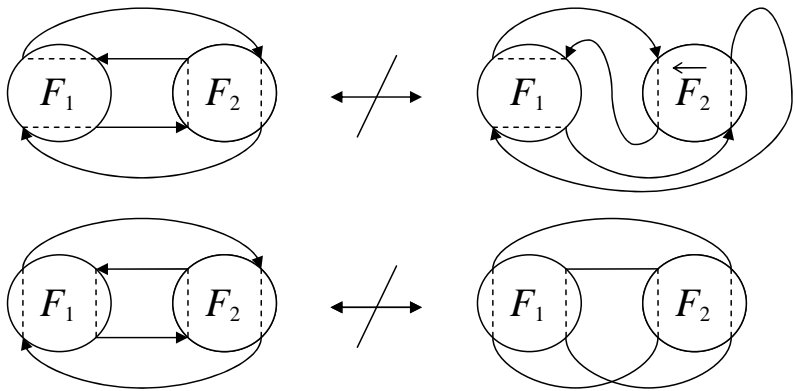}%
\caption{Euler systems need not correspond under \textquotedblleft unbalanced
mutation.\textquotedblright}%
\label{tranmf19}%
\end{center}
\end{figure}
%EndExpansion

Before proceeding, we should emphasize that the correspondence between Euler
systems illustrated in Fig. \ref{tranmf18} is peculiar to balanced mutations.
Graphs that are not quite balanced mutations do not exhibit the same
correspondence. See Fig. \ref{tranmf19}, which displays two examples of graphs
related through reassembly of four-element edge cuts. The first example is not
a balanced mutation because the edges do not correspond in pairs, and the
second is not a balanced mutation because there are only two new edges rather
than four. In each example, if the dashed walks within $F_{1}$ or $F_{2}$ are
not vertex-disjoint then one of the indicated circuit partitions may be an
Euler system, but the other cannot.

\section{Touch-graphs and a formula of Las Vergnas and Martin}

Suppose $P$ is a circuit partition in a 4-regular graph $F$. As mentioned in
the introduction, Bouchet \cite{Bi1} discussed a way to construct a
touch-graph $Tch(P)$ from $P$, with a vertex for each circuit of $P$ and an
edge for each vertex of $F$; the edge corresponding to $v$ is incident on the
vertex or vertices corresponding to circuit(s) of $P$ incident at $v$. (This
construction appeared implicitly in earlier work of Jaeger \cite{J1}.)

\begin{theorem}
\label{cycletouch}Let $F$ and $F^{\prime}$ be 4-regular graphs with isomorphic
transition matroids. Then\ some isomorphism $f:M_{\tau}(F)\rightarrow M_{\tau
}(F^{\prime})$ has the following properties:

\begin{enumerate}
\item There is a bijection $g:V(F)\rightarrow V(F^{\prime})$ that is
compatible with $f$, in the sense that for each $v\in V(F)$, the images under
$f\,$\ of the three transitions at $v$ are the three transitions of
$F^{\prime}$ at $g(v)$.

\item If $P$ and $P^{\prime}$ are circuit partitions with $\tau(P^{\prime
})=f(\tau(P))$, then $g$ defines an isomorphism between the cycle matroids of
$Tch(P)$ and $Tch(P^{\prime})$.
\end{enumerate}
\end{theorem}

\begin{proof}
By Theorem \ref{iso}, it suffices to consider the possibility that $F^{\prime
}$ is obtained from $F$ by a single connected sum, separation or balanced
mutation. Then $F$ and $F^{\prime}$ have the same vertices and the same
half-edges, though the half-edges displayed in Fig. \ref{tranmf5a} or
\ref{tranmf5b1} (whichever figure is appropriate) are combined into edges in
different ways. As transitions are pairings of half-edges incident at the same
vertex, $F$ and $F^{\prime}$ also have the same transitions. If we use $g$ and
$f$ to denote the identity maps of $V(F)=V(F^{\prime})$ and $\mathfrak{T}%
(F)=\mathfrak{T}(F^{\prime})$, respectively, then these identity maps
certainly satisfy assertion 1.

If $C$ is an Euler system of $F$ then as discussed in the proof of Theorem
\ref{ghier2}, $F^{\prime}$ has an Euler system $C^{\prime}$ such that $W(C)$
and $W(C^{\prime})$ differ only by equivalences and turnarounds. Then the
interlacement graphs $\mathcal{I}(C)$ and $\mathcal{I}(C^{\prime})$ are the
same, so the identity map $f$ defines an isomorphism of transition matroids.

It remains to verify assertion 2. Each circuit partition $P$ of $F$
corresponds to a circuit partition $P^{\prime}$ of $F^{\prime}$, with
$\tau(P^{\prime})=f(\tau(P))$. The edges of $Tch(P)$ and $Tch(P^{\prime})$
correspond to elements of $V(F)=V(F^{\prime})$, so the identity map $g$ is a
bijection between the edge sets of $Tch(P)$ and $Tch(P^{\prime})$. If
$F^{\prime}$ is obtained from $F$ by a\ connected sum of $F_{1}$ and $F_{2}$,
then $P^{\prime}$ is obtained by \textquotedblleft splicing\textquotedblright%
\ two circuits of $P$, one contained in $F_{1}$ and the other contained in
$F_{2}$. Clearly then $Tch(P^{\prime})$ is obtained from $Tch(P)$ by
identifying two vertices from different connected components, so $Tch(P)$ and
$Tch(P^{\prime})$ have the same cycle matroid. If $F^{\prime}$ is obtained
from $F$ by a\ separation, the situation is reversed.

Suppose now that $F^{\prime}$ is obtained from $F$ by a balanced mutation
which is not a composition of connected sums and separations. Then the four
edges connecting $F_{1}$ to $F_{2}$ lie in the same connected component of
$F$. Let $P~$and $P^{\prime}$ be circuit partitions of $F$ and $F^{\prime}$,
respectively, with $\tau(P^{\prime})=f(\tau(P))$. As the only difference
between $F$ and $F^{\prime}$ is that different half-edges constitute\ the four
edges connecting $F_{1}$ to $F_{2}$, the only differences between $P$ and
$P^{\prime}$ involve the circuit(s) containing the four edges that connect
$F_{1}$ to $F_{2}$. In particular, the circuits of $P$ that are contained in
$F_{1}$ or $F_{2}$ are the same as the circuits of $P^{\prime}$ that are
contained in $F_{1}$ or $F_{2}$.

We claim that there is a bijection $f_{P}:P\rightarrow P^{\prime}$ that is the
identity map for circuits contained in $F_{1}$ or $F_{2}$. If the claim is
false then we may suppose without loss of generality that the four edges
connecting $F_{1}$ to $F_{2}$ appear in one circuit of $P$, but in two
distinct circuits of $P^{\prime}$. We use induction on $\left\vert
P\right\vert $ to prove that this is impossible. If $\left\vert P\right\vert
=c(F)$ then $P$ is an Euler system of $F$. As illustrated in Fig.
\ref{tranmf18}, it follows that $P^{\prime}$ is an Euler system of $F^{\prime
}$, and all four edges connecting $F_{1}$ to $F_{2}$ in $F^{\prime}$ lie in
one connected component; but then they all appear in one circuit of
$P^{\prime}$, a contradiction. Proceeding inductively, suppose $\left\vert
P\right\vert >c(F)$ and the claim holds for all circuit partitions of size
smaller than $\left\vert P\right\vert $. As the four edges connecting $F_{1}$
to $F_{2}$ lie on a single circuit of $P$, $\left\vert P\right\vert >c(F)$
implies that there is a circuit $\gamma_{1}\in P$ that is contained in either
$F_{1}$ or $F_{2}$, and is not an Euler circuit of a connected component of
$F$. Consequently there is a vertex $v$ where $\gamma_{1}$ and some other
circuit $\gamma_{2}\in P$ are both incident. As noted earlier, $P$ and
$P^{\prime}$ include the same circuits contained in $F_{1}$ or $F_{2}$, so
$\gamma_{1}\in P^{\prime}$ and $P^{\prime}$ also includes a circuit
$\gamma_{2}^{\prime}\neq\gamma_{1}$ that is incident at $v$. Let $\overline
{P}$ be a circuit partition of $F$ obtained by changing the transition of $P$
at $v$, and let $\overline{P^{\prime}}$ be the corresponding circuit partition
of $F^{\prime}$. As indicated in Fig. \ref{tranmf10}, $\left\vert \overline
{P}\right\vert =\left\vert P\right\vert -1$; consequently the inductive
hypothesis tells us that all four edges connecting $F_{1}$ to $F_{2}$ in
$F^{\prime}$ lie on a single circuit $\overline{\gamma}$ of $\overline
{P^{\prime}}$. By hypothesis $P^{\prime}$ does not contain any such circuit,
so $\overline{\gamma}\in\overline{P^{\prime}}-P^{\prime}$, i.e.,
$\overline{\gamma}$ is the circuit obtained by uniting $\gamma_{1}$ and
$\gamma_{2}^{\prime}$. As $\gamma_{1}$ is contained in $F_{1}$ or $F_{2}$,
$\gamma_{2}^{\prime}$ must contain all four edges connecting $F_{1}$ to
$F_{2}$ in $F^{\prime}$. But again, by hypothesis $P^{\prime}$ does not
contain any such circuit. This contradiction establishes the claim.%
%TCIMACRO{\FRAME{ftbpFU}{4.7928in}{0.5889in}{0pt}{\Qcb{Three circuit partitions
%that differ at only one vertex.}}{\Qlb{tranmf10}}{tranmf10.ps}%
%width 4.7928in;  height 0.5889in;  depth 0pt;  original-width 8.4968in;
%original-height 11.0056in;  cropleft "0.1731";  croptop "0.8842";
%cropright "0.9208";  cropbottom "0.8162";
%filename 'tranmf10.ps';file-properties "XNPEU";}} }%
%BeginExpansion
\begin{figure}
[ptb]
\begin{center}
\includegraphics[
trim=1.470796in 8.982771in 0.672947in 1.274449in,
height=0.5889in,
width=4.7928in
]%
{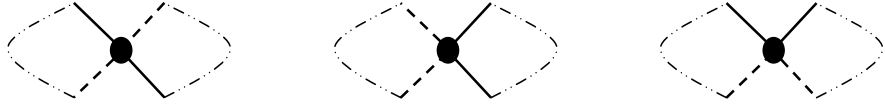}%
\caption{Three circuit partitions that differ at only one vertex.}%
\label{tranmf10}%
\end{center}
\end{figure}
%EndExpansion

As $P=V(Tch(P))$ and $P^{\prime}=V(Tch(P^{\prime}))$, the claim tells us that
$f_{P}$ is a bijection between the vertex sets of $Tch(P)$ and $Tch(P^{\prime
})$. The relationship between the two touch-graphs falls into one of the
following cases. Case 1: The four edges connecting $F_{1}$ to $F_{2}$ in $F$
all appear in one circuit of $P$. In this case the four edges connecting
$F_{1}$ to $F_{2}$ in $F^{\prime}$ all appear in one circuit of $P^{\prime}$,
and $f_{P}$ defines an isomorphism $Tch(P)\cong Tch(P^{\prime})$. Case 2: The
four edges connecting $F_{1}$ to $F_{2}$ in $F$ appear in two different
circuits $\gamma_{1}\neq\gamma_{2}\in P$, and the four edges connecting
$F_{1}$ to $F_{2}$ in $F^{\prime}$ appear in two different circuits
$f_{P}(\gamma_{1})\neq f_{P}(\gamma_{2})\in P^{\prime}$. In this case
$f_{P}(\gamma_{1})$ and $f_{P}(\gamma_{2})$ are the circuits of $F^{\prime}$
obtained by \textquotedblleft splicing\textquotedblright\ the portions of
$\gamma_{1}$ and $\gamma_{2}$ within $F_{1}$ and $F_{2}$, using the four edges
of $F^{\prime}$ that connect $F_{1}$ and $F_{2}$. Interchanging $f_{P}%
(\gamma_{1})$ and $f_{P}(\gamma_{2})$ if necessary, we may presume that
$f_{P}$ has been defined in such a way that $f_{P}(\gamma_{1})$ contains the
$F_{1}$ portion of $\gamma_{1}$.\ Subcase 2a: The two portions of $\gamma_{1}%
$\ are spliced to form $f_{P}(\gamma_{1})$ and the two portions of $\gamma
_{2}$\ are spliced to form $f_{P}(\gamma_{2})$. In this case, $f_{P}$ defines
an isomorphism $Tch(P)\cong Tch(P^{\prime})$. Subcase 2b: The $F_{1}$ portion
of $\gamma_{1}$ and the $F_{2}$ portion of $\gamma_{2}$ are spliced to form
$f_{P}(\gamma_{1})$, and $f_{P}(\gamma_{2})$ is formed by splicing the $F_{1}$
portion of $\gamma_{2}$ and the $F_{2}$ portion of $\gamma_{1}$. In this case
the bijection $f_{P}:V(Tch(P))\rightarrow V(Tch(P^{\prime}))$ preserves all
adjacencies, except that each edge in $Tch(P)$ between $\gamma_{1}$
(respectively, $\gamma_{2}$) and a circuit $\gamma$ contained in $F_{2}$ is
replaced in $Tch(P^{\prime})$ with an edge between $f_{P}(\gamma_{2})$
(respectively, $f_{P}(\gamma_{1})$) and $\gamma$. If no circuit of $P$ is
contained in $F_{2}$ then this edge-swapping is immaterial and $f_{P}$ defines
an isomorphism $Tch(P)\cong Tch(P^{\prime})$. If no circuit of $P$ is
contained in $F_{1}$ then this edge-swapping tells us that an isomorphism
$Tch(P)\cong Tch(P^{\prime})$ is given by the bijection obtained from $f_{P}$
by interchanging the values of $f_{P}(\gamma_{1})$ and $f_{P}(\gamma_{2})$. If
each of $F_{1}$, $F_{2}$ contains a circuit of $P$ then $\{\gamma_{1}%
,\gamma_{2}\}$ is a vertex cut of $Tch(P)$, which separates the circuit(s)
contained in $F_{1}$ from the circuit(s) contained in $F_{2}$. The
edge-swapping is then a Whitney twist with respect to $\{\gamma_{1},\gamma
_{2}\}$.

In every case $Tch(P)$ and $Tch(P^{\prime})$ have isomorphic cycle matroids,
and the matroid isomorphisms respect the identification of touch-graph edges
with elements of $V(F)=V(F^{\prime})$. Consequently assertion 2 holds.
\end{proof}

Examples of subcases 2a and 2b appear in Figs. \ref{tranmf17a} and
\ref{tranmf17b}, respectively. Transitions are indicated in the figures by
plain, dashed and bold line styles; when a circuit traverses a vertex the
style is maintained, but sometimes it is necessary to change the style in the
middle of an edge. To indicate the correspondence between circuits and
touch-graph vertices, some circuits and vertices are numbered; also, every
circuit with dashed transitions is represented by a dashed vertex.%
%TCIMACRO{\FRAME{fpFU}{4.9908in}{3.071in}{0pt}{\Qcb{Corresponding circuit
%partitions in balanced mutations, whose touch-graphs are the same.}%
%}{\Qlb{tranmf17a}}{tranmf17a.ps}{\special{ language "Scientific Word";
%type "GRAPHIC";  maintain-aspect-ratio TRUE;  display "USEDEF";
%valid_file "F";  width 4.9908in;  height 3.071in;  depth 0pt;
%original-width 8.4968in;  original-height 11.0056in;  cropleft "0.1417";
%croptop "0.9087";  cropright "0.9205";  cropbottom "0.5404";
%filename 'tranmf17a.ps';file-properties "XNPEU";}} }%
%BeginExpansion
\begin{figure}
[p]
\begin{center}
\includegraphics[
trim=1.203997in 5.947426in 0.675496in 1.004811in,
height=3.071in,
width=4.9908in
]%
{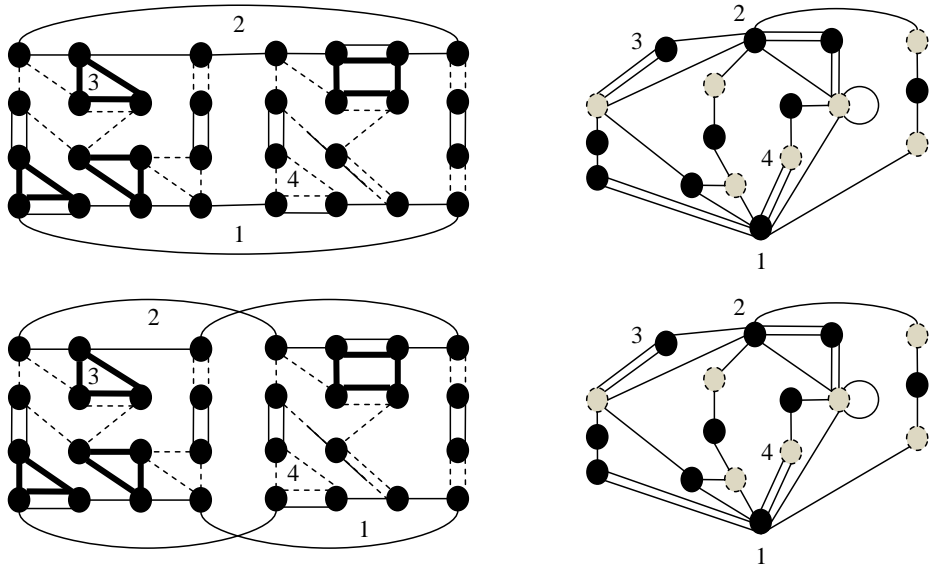}%
\caption{Corresponding circuit partitions in balanced mutations, whose
touch-graphs are the same.}%
\label{tranmf17a}%
\end{center}
\end{figure}
%EndExpansion
%

%TCIMACRO{\FRAME{fpFU}{5.0946in}{3.0623in}{0pt}{\Qcb{Corresponding circuit
%partitions in balanced mutations, whose touch-graphs are related by a Whitney
%twist.}}{\Qlb{tranmf17b}}{tranmf17b.ps}{\special{ language "Scientific Word";
%type "GRAPHIC";  maintain-aspect-ratio TRUE;  display "USEDEF";
%valid_file "F";  width 5.0946in;  height 3.0623in;  depth 0pt;
%original-width 8.4968in;  original-height 11.0056in;  cropleft "0.1417";
%croptop "0.9024";  cropright "0.9368";  cropbottom "0.5347";
%filename 'tranmf17b.ps';file-properties "XNPEU";}} }%
%BeginExpansion
\begin{figure}
[p]
\begin{center}
\includegraphics[
trim=1.203997in 5.884694in 0.536998in 1.074146in,
height=3.0623in,
width=5.0946in
]%
{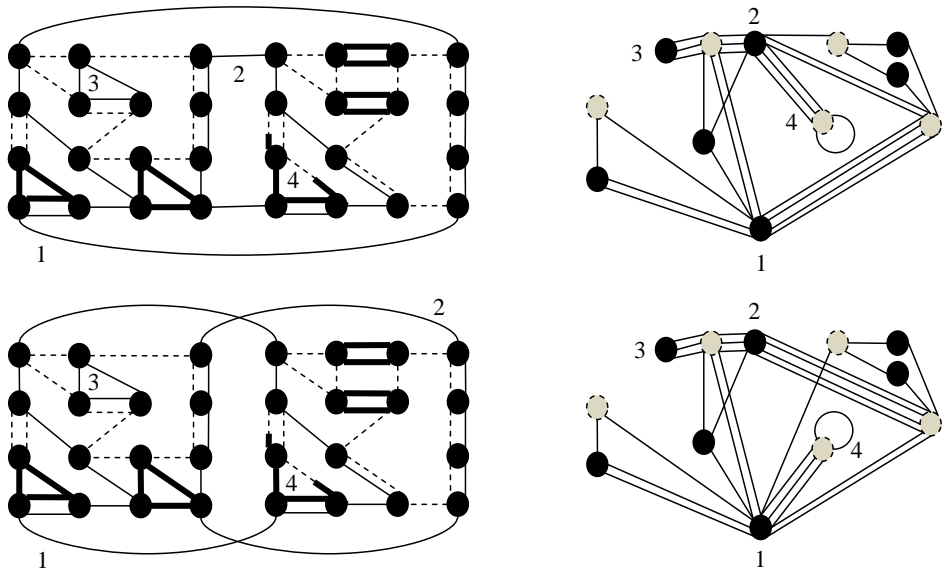}%
\caption{Corresponding circuit partitions in balanced mutations, whose
touch-graphs are related by a Whitney twist.}%
\label{tranmf17b}%
\end{center}
\end{figure}
%EndExpansion

If $F$ is a 4-regular graph, then Theorem \ref{cycletouch} tells us indirectly
that $M_{\tau}(F)$ determines the cycle matroids of all the touch-graphs of
circuit partitions in $F$. There is a more direct way to see this property,
using matroid duality. Here are three ways to describe the dual $M^{\ast}$ of
a binary matroid $M$ on a set $S$:

\begin{enumerate}
\item The bases of $M^{\ast}$ are the complements of the bases of $M$.

\item The rank functions $r$ of $M$ and $r^{\ast}$ of $M^{\ast}$ are related
by the fact that for every $A\subseteq S$, the nullity of $A$ in $M^{\ast}$
equals the corank of $S-A$ in $M$:%
\[
\left\vert A\right\vert -r^{\ast}(A)=r(S)-r(S-A).
\]

\item The subspace of $GF(2)^{S}$ spanned by the rows of a matrix representing
$M$ is the orthogonal complement of the subspace spanned by the rows of a
matrix representing $M^{\ast}$.
\end{enumerate}

Let $P$ be a circuit partition of $F$, and recall that $\tau(P)$ is the subset
of $\mathfrak{T}(F)$ consisting of the transitions involved in $P$; let
$M_{\tau}(P)$ denote the matroid obtained by restricting $M_{\tau}(F)$ to
$\tau(P)$. Notice that both $M_{\tau}(P)$ and the cycle matroid of $Tch(P)$
have elements corresponding to the vertices of $F$. It turns out that these
matroids are closely connected:

\begin{theorem}
\label{core}The cycle matroid of the graph $Tch(P)$ and the dual of $M_{\tau
}(P)$ define the same matroid on $V(F)$.
\end{theorem}

We do not provide a proof of Theorem \ref{core} because several equivalent
results have already appeared in the literature. An algebraic form of Theorem
\ref{core} is this:

\begin{theorem}
\label{core2}For any Euler system $C$ of $F$, let $M(C,P)$ be the submatrix of
$M(C)$ consisting of columns corresponding to elements of $\tau(P)$. Then the
orthogonal complement of the row space of $M(C,P)$ is the subspace of
$GF(2)^{V(F)}$ spanned by the elementary cocycles of $Tch(P)$.
\end{theorem}

Here the elementary cocycle corresponding to a vertex of a graph $G$ is the
subset of $E(G)$ consisting of the non-loop edges incident at that vertex.
Jaeger \cite{J1} verified Theorem \ref{core2} in case $P$ does not involve the
$\phi_{C}(v)$ transition at any vertex, and Bouchet \cite{Bu} gave a different
argument for the particular case in which $P$ involves only $\chi_{C}(v)$
transitions. The general case of Theorem \ref{core2} is proven in \cite{Tnew}.
A detailed discussion of Theorem \ref{core} is given in \cite{BHT}; like
Jaeger's original version the argument of \cite{BHT} requires that $P$ not
involve the $\phi_{C}(v)$ transition at any vertex, but Proposition \ref{inv}
renders the requirement moot. Yet another equivalent version of Theorem
\ref{core} appears in Bouchet's work on isotropic systems; see Section 6 of
\cite{Bi1}.

Theorem \ref{core} tells us that the duals of the cycle matroids of the
touch-graphs of all the\ different circuit partitions of $F$ are interwoven in
the single matroid $M_{\tau}(F)$. This property might well remind the reader
of Bouchet's isotropic systems \cite{Bi1, Bi2} or the cycle family graphs of
Ellis-Monaghan and Moffatt \cite{EMM, EMM1, EMM2}; indeed, transition matroids
are implicit in both these sets of ideas.

Theorem \ref{core} also provides a new proof of a famous formula introduced by
Martin \cite{Ma} and explained and extended by Las Vergnas \cite{L2}, which
ties the Martin polynomial of a 4-regular plane graph to the diagonal Tutte
polynomial of either of its associated \textquotedblleft
checkerboard\textquotedblright\ graphs. The formula is given in Corollary
\ref{Martin2} below.

\begin{theorem}
\label{Martin1} Suppose $F$ is a 4-regular graph with $n$ vertices. Let
$P_{1}$ and $P_{2}$ be circuit partitions in $F$, and for $v\in V(F)$ let
$\tau_{1}(v),\tau_{2}(v)$ be the transitions involved in $P_{1}$ and $P_{2}$,
respectively. Suppose $\tau_{1}(v)\neq\tau_{2}(v)$ $\forall v\in V(F)$. Then
if $B_{1}$ is any basis of $M_{\tau}(P_{1})$,
\[
B_{1}\cup\{\tau_{2}(v)\mid\tau_{1}(v)\notin B_{1}\}
\]
is a basis of $M_{\tau}(F)$. It follows that
\[
r(\tau(P_{1})\cup\tau(P_{2}))=n\leq r(\tau(P_{1}))+r(\tau(P_{2})).
\]

\end{theorem}

\begin{proof}
Let $\{v_{1},...,v_{p}\}=\{v\in V(F)\mid\tau_{1}(v)\notin B_{1}\}$. For $0\leq
i\leq p$ let $P_{i}^{\prime}$ be the circuit partition of $F$ such that%
\[
\tau(P_{i}^{\prime})=\{\tau_{1}(v)\mid v\notin\{v_{1},...,v_{i}\}\}\cup
\{\tau_{2}(v_{1}),...,\tau_{2}(v_{i})\}\text{.}%
\]

We claim that $\left\vert P_{i}^{\prime}\right\vert =\left\vert P_{1}%
\right\vert -i$ for each $i$; as $P_{0}^{\prime}=P_{1}$, the claim is true for
$i=0$. Suppose $i\geq1$ and the claim holds for $P_{0}^{\prime},...,P_{i-1}%
^{\prime}$. The three circuit partitions that involve the same transitions as
$P_{i-1}^{\prime}$ at all vertices other than $v_{i}$ are pictured in Fig.
\ref{tranmf10}. Suppose $P_{i-1}^{\prime}$ is not the partition pictured in
the middle. Let $\tau(v_{i})$ be the transition of the circuit partition
pictured in the middle, and let $P^{\prime\prime}$ be the circuit partition
with $\tau(P^{\prime\prime})=(\tau(P_{i-1}^{\prime})-\{\tau_{1}(v_{i}%
)\})\cup\{\tau(v_{i})\}$. Then $\left\vert P^{\prime\prime}\right\vert
=\left\vert P_{i-1}^{\prime}\right\vert +1$, so the circuit-nullity formula
tells us that $r(\tau(P^{\prime\prime}))=r(\tau(P_{i-1}^{\prime}))-1$.
Consequently, adjoining $\tau_{1}(v_{i})$ to $\tau(P^{\prime\prime})$ raises
the rank. This is impossible, though, because $B_{1}\subseteq\tau
(P^{\prime\prime})$ and $B_{1}$ spans $\tau(P_{1})$. We conclude that
$P_{i-1}^{\prime}$ is the partition pictured in the middle of Fig.
\ref{tranmf10}, so $P_{i}^{\prime}$ is not; hence $\left\vert P_{i}^{\prime
}\right\vert =\left\vert P_{i-1}^{\prime}\right\vert -1$, and the claim holds
for $P_{i}^{\prime}$.

The circuit-nullity formula now tells us that $r(\tau(P_{i}^{\prime
}))=i+r(\tau(P_{1}))$ for each $i$. In particular, $r(\tau(P_{p}^{\prime
}))=p+r(\tau(P_{1}))=n-\left\vert B_{1}\right\vert +r(\tau(P_{1}))$. As
$B_{1}$ is a basis of $M_{\tau}(P_{1})$, $\left\vert B_{1}\right\vert
=r(\tau(P_{1}))$ and hence $r(\tau(P_{p}^{\prime}))=n$. As $\tau(P_{p}%
^{\prime})\subseteq\tau(P_{1})\cup\tau(P_{2})$, we conclude that $r(\tau
(P_{1})\cup\tau(P_{2}))=n$.

The proof is completed by recalling that $r(A_{1}\cup A_{2})\leq
r(A_{1})+r(A_{2})$ for any two subsets of any matroid.
\end{proof}

If a subset $A\subseteq V(F)$ has the property that $\tau_{2}(A)=\{\tau
_{2}(v)\mid v\in A\}$ is a circuit in $M_{\tau}(P_{2})$, then the
contrapositive of Theorem \ref{Martin1} implies that $\{\tau_{1}(v)\mid
v\notin A\}$ does not contain any basis of $M_{\tau}(P_{1})$. Consequently
$\tau_{1}(A)=\{\tau_{1}(v)\mid v\in A\}$ is not contained in any basis of
$M_{\tau}(P_{1})^{\ast}$; that is, $\tau_{1}(A)$ contains a circuit of
$M_{\tau}(P_{1})^{\ast}$. This property -- if $\tau_{2}(A)$ is a circuit in
$M_{\tau}(P_{2})$ then $\tau_{1}(A)$ contains a circuit of $M_{\tau}%
(P_{1})^{\ast}$ -- is expressed by saying that the identity map of $V(F)$
defines a weak map from $M_{\tau}(P_{2})$ to $M_{\tau}(P_{1})^{\ast}$. The
weak map property is sufficient for our present purposes, but it is worth
noting that in fact, the identity map of $V(F)$ defines a strong map from
$M_{\tau}(P_{2})$ to $M_{\tau}(P_{1})^{\ast}$. That is, if $\tau_{2}(A)$ is a
circuit in $M_{\tau}(P_{2})$ then $\tau_{1}(A)$ is a disjoint union of
circuits of $M_{\tau}(P_{1})^{\ast}$. (Edmonds \cite{E} proved an equivalent
result, in which $Tch(P_{1})$ and $Tch(P_{2})$ are geometrically dual embedded
graphs on a surface; the fact that Edmonds' result yields a strong map was
mentioned by Las Vergnas \cite{L1}.) To see why this is true, consider that
the circuit-nullity formula tells us that a circuit of $M_{\tau}(P_{2})$
corresponds to a minimal circuit of $F$ involving only transitions from
$P_{2}$; say the circuit is $v_{1}$, $h_{1}$, $h_{1}^{\prime}$, $v_{2}$, ...,
$h_{\ell-1}$, $h_{\ell-1}^{\prime}=h_{0}^{\prime}$, $v_{\ell}=v_{1}$. As
$P_{1}$ and $P_{2}$ do not involve the same transition at any vertex, none of
the single transitions $h_{i-1}^{\prime}$, $v_{i}$, $h_{i}$ appears in $P_{1}%
$. Hence each of them corresponds to an edge of $Tch(P_{1})$, and the circuit
corresponds to a closed walk in $Tch(P_{1})$.\ A closed walk is a union of
edge-disjoint minimal circuits, of course, and the minimal circuits of
$Tch(P_{1})$ correspond to circuits of its cycle matroid. Finally, Theorem
\ref{core} tells us that the identity map of $V(F)$ defines an isomorphism
between the cycle matroid of $Tch(P_{1})$ and $M_{\tau}(P_{1})^{\ast}$.

\begin{corollary}
\label{Martin} In the situation of Theorem \ref{Martin1}, these conditions are equivalent:

\begin{enumerate}
\item $r(\tau(P_{1}))+r(\tau(P_{2}))=n$.

\item Each of $M_{\tau}(P_{1}),M_{\tau}(P_{2})$ is\ isomorphic to the dual of
the other.

\item $M_{\tau}(P_{1})$ and $M_{\tau}(P_{2})$ define dual matroids on $V(F)$.

\item The matroid obtained by restricting $M_{\tau}(F)$ to $\tau(P_{1}%
)\cup\tau(P_{2})$ is the direct sum $M_{\tau}(P_{1})\oplus M_{\tau}(P_{2})$.
\end{enumerate}
\end{corollary}

\begin{proof}
As the rank of a direct sum is simply the total of the ranks of the summands,
the implication $4\Rightarrow1$ follows directly from the fact that
$n=r(\tau(P_{1})\cup\tau(P_{2}))$. For the opposite implication, suppose
$r(\tau(P_{1}))+r(\tau(P_{2}))=n$. If there were any nontrivial intersection
between the linear span of the columns of $M(C)$ corresponding to $\tau
(P_{1})$ and the linear span of the columns corresponding to $\tau(P_{2})$, we
would have $r(\tau(P_{1}))+r(\tau(P_{2}))>r(\tau(P_{1})\cup\tau(P_{2}))$;
$r(\tau(P_{1})\cup\tau(P_{2}))=n$, so there is no such nontrivial
intersection. We conclude that $1\Rightarrow4$.

The implication $3\Rightarrow2$ is obvious, and $2\Rightarrow1$ follows
immediately from any of the descriptions of dual matroids mentioned at the
beginning of the section.

It remains to verify the implication $1\Rightarrow3$. Suppose $r(\tau
(P_{1}))+r(\tau(P_{2}))=n$, and let $B_{1}=\{\tau_{1}(a)\mid a\in A\}$ and
$B_{2}=\{\tau_{2}(v)\mid v\notin A\}$. If $B_{1}$ is a basis of $M_{\tau
}(P_{1})$, then according to Theorem \ref{Martin1}, $B_{1}\cup B_{2}$ is an
independent set of $M_{\tau}(F)$, so $B_{2}$ is an independent set of
$M_{\tau}(P_{2})$. As $\left\vert B_{2}\right\vert =n-\left\vert
B_{1}\right\vert $ is the rank of $M_{\tau}(P_{2})$, $B_{2}$ must be a basis
of $M_{\tau}(P_{2})$. Similarly, if $B_{2}$ is a basis of $M_{\tau}(P_{2})$
then $B_{1}$ is a basis of $M_{\tau}(P_{1})$.
\end{proof}

Notice that Theorem \ref{core} implies that if the equivalent conditions
of\ Corollary \ref{Martin} hold then $Tch(P_{1})$ and $Tch(P_{2})$ are dual
graphs, and hence must be planar. It turns out that in this situation $F$ must
be planar too; see Section 9 for details.

\begin{corollary}
\label{Martin2}\cite{L2, Ma} Suppose the equivalent conditions of Corollary
\ref{Martin} hold, and let $r_{T}$ be the rank function of the cycle matroid
of $Tch(P_{1})$. Let $A\subseteq V(F)$ be any subset, and let $P_{A}$ be the
circuit partition of $F$ with $\tau(P_{A})=\{\tau_{1}(a)\mid a\in
A\}\cup\{\tau_{2}(v)\mid v\in V(F)-A\}$. Then $\left\vert P_{A}\right\vert
-c(F)=r_{T}(V(F))+\left\vert A\right\vert -2r_{T}(A)$.
\end{corollary}

\begin{proof}
The circuit-nullity formula tells us that $\left\vert P_{A}\right\vert
-c(F)=n-r(\tau(P_{A}))$, where $r$ is the rank function of $M_{\tau}(F)$. As
$\tau(P_{A})\subseteq\tau(P_{1})\cup\tau(P_{2})$, part 4 of Corollary
\ref{Martin} tells us that
\[
r(\tau(P_{A}))=r(\tau(P_{A})\cap\tau(P_{1}))+r(\tau(P_{A})\cap\tau
(P_{2}))\text{.}%
\]
Part 3 of Corollary \ref{Martin} tells us that the corank of $\tau(P_{A}%
)\cap\tau(P_{1})$ in $M_{\tau}(P_{1})$ is the same as the nullity of
$\tau(P_{A})\cap\tau(P_{2})$ in $M_{\tau}(P_{2})$, i.e.,
\[
r(\tau(P_{1}))-r(\tau(P_{A})\cap\tau(P_{1}))=n-\left\vert A\right\vert
-r(\tau(P_{A})\cap\tau(P_{2}))\text{.}%
\]
It follows that
\begin{align*}
\left\vert P_{A}\right\vert -c(F)  &  =n-r(\tau(P_{A}))=n-r(\tau(P_{A}%
)\cap\tau(P_{1}))-r(\tau(P_{A})\cap\tau(P_{2}))\\
&  =r(\tau(P_{1}))-r(\tau(P_{A})\cap\tau(P_{1}))+\left\vert A\right\vert
-r(\tau(P_{A})\cap\tau(P_{1})).
\end{align*}
This is the sum of the corank and the nullity of $\tau(P_{A})\cap\tau(P_{1})$
in $M_{\tau}(P_{1})$, so it equals the sum of the nullity and the corank of
$\tau(P_{A})\cap\tau(P_{1})$ in the dual matroid of $M_{\tau}(P_{1})$. The
formula of the statement follows, because Theorem \ref{core} tells us that the
cycle matroid of $Tch(P_{1})$ and the dual matroid of $M_{\tau}(P_{1})$ define
the same matroid on $V(F)$.
\end{proof}

\section{Proof of Theorem \ref{MTutte}}

In this section we prove Theorem \ref{MTutte} of the introduction, which
asserts that if we are given $c(F)$, then the parametrized Tutte polynomial of
$M_{\tau}(F)$ determines the directed and undirected Martin polynomials of
$F$. The idea of the proof is that we choose parameters for which every
nonzero parameter product in the parametrized Tutte polynomial $f(M_{\tau
}(F))$ corresponds to a circuit partition.

Begin with a ring $R$ of polynomials in $6\left\vert V(F)\right\vert +2$
indeterminates, $x$ and $y$ in addition to $\alpha(\tau)$ and $\beta(\tau)$
for each transition $\tau\in\mathfrak{T}(F)$. Let $J$ be the ideal of $R$
generated by all products $\alpha(\tau_{1})\alpha(\tau_{2})$, $\alpha(\tau
_{1})\alpha(\tau_{3})$, $\alpha(\tau_{2})\alpha(\tau_{3})$ and $\beta(\tau
_{1})\beta(\tau_{2})\beta(\tau_{3})$ such that $\tau_{1},\tau_{2},\tau_{3}$
are the three transitions corresponding to a single vertex. Consider the
parametrized Tutte polynomial $f(M_{\tau}(F))$ in the quotient ring $R/J$.
Then the nonzero terms of $f(M_{\tau}(F))$ correspond to subsets
$A\subset\mathfrak{T}(F)$ that include precisely one transition for each
vertex of $F$. Each such subset $A$ has $|A|=n=r(\mathfrak{T}(F))$, so Theorem
\ref{tranmat} tells us that%
\[
f(M_{\tau}(F))=\sum_{P}\left(  \prod\limits_{\tau\in\tau(P)}\alpha
(\tau)\right)  \left(  \prod\limits_{\tau\notin\tau(P)}\beta(\tau)\right)
(x-1)^{\left\vert P\right\vert -c(F)}(y-1)^{\left\vert P\right\vert
-c(F)}\text{,}%
\]
with a term for each circuit partition $P$ of $F$.

Note that if $\pi:R\rightarrow R/J$ is the canonical map onto the quotient,
then $\pi$ is injective on the additive subgroup of $R$ generated by monomials
that include no more than one of the $\alpha(\tau)$ corresponding to any
vertex and no more than two of the $\beta(\tau)$ corresponding to any vertex.
The version of the Martin polynomial used by Las Vergnas \cite{L}, $\sum
_{P}(\zeta-1)^{\left\vert P\right\vert -1}$, may be obtained from the inverse
image $\pi^{-1}f(M_{\tau}(F))$ by setting $\alpha(\tau)=1$ and $\beta(\tau)=1$
for every transition $\tau$, setting $x=\zeta$ and $y=2$, and multiplying by
$(\zeta-1)^{c(F)-1}$.

To obtain a directed Martin polynomial, enlarge $J$ by including $\alpha
(\tau)$ for every direction-violating transition $\tau$.

\section{4-regular graphs and ribbon graphs}

In this section we provide a brief exposition of some ideas from topological
graph theory. For thorough discussions we refer to the literature \cite{BR1,
BR2, Ch, EMM, EMM1, EMM2, GT, MT}.

A ribbon graph is a graph $G$ given with additional information regarding the
edges and vertices of $G$. Each vertex $v\in V(G)$ is given with a prescribed
order of the half-edges incident at $v$; the collection of these orders is
called a rotation system on $G$. Also, each edge $e\in E(G)$ is labeled $1$ or
$-1$. Two ribbon graphs are equivalent if there is a graph isomorphism between
them, such that the rotation systems and $\pm1$ labels are related by a
sequence of operations of these two types: (a) cyclically permute the
half-edge order at a vertex, leaving the $\pm1$ labels unchanged, and (b)
reverse the half-edge order at a vertex, reversing all the $\pm1$ labels of
non-loop edges incident at that vertex (n.b. loops retain their $\pm1$ labels).

The definition of equivalence is motivated by thinking of a ribbon graph $G$
as a blueprint for constructing a surface with boundary $S(G)$. The first step
of the construction is to replace each vertex with a disk, whose boundary
carries a preferred orientation. The second step is to choose, for each
half-edge of $G$, two points on the boundary of the disk representing the
vertex incident on that half-edge; the two points that represent a half-edge
should be close together, and separated from those representing any other
half-edge. Moreover, these pairs of points should be chosen so that as one
walks around the boundary circle in the preferred direction, one encounters
the pairs in the prescribed order of the corresponding half-edges. The third
step is to replace each edge with a narrow band, whose two ends correspond to
the two half-edges of the edge. Each end of the band is connected to the disk
corresponding to the vertex incident on that half-edge, by identifying the end
with the short arc on the disk's boundary bounded by the points corresponding
to that half-edge. If the edge is labeled $1$ then the band-ends are attached
so that the band's boundary may be oriented consistently with the preferred
orientation(s) of the incident disk(s); if the edge is labeled $-1$ then the
band-ends are attached so that the band's boundary cannot be oriented
consistently with the preferred orientation(s). See Fig. \ref{tranmf6}, where
bands representing edges labeled $1$ are on the left, and bands representing
edges labeled $-1$ are on the right.%
%TCIMACRO{\FRAME{fpFU}{3.5639in}{1.7115in}{0pt}{\Qcb{Disks and bands
%representing edges labeled $1$ and $-1$, respectively.}}{\Qlb{tranmf6}%
%}{tranmf6.ps}{\special{ language "Scientific Word";  type "GRAPHIC";
%width 3.5639in;  height 1.7115in;  depth 0pt;  original-width 8.4968in;
%original-height 11.0056in;  cropleft "0.2356";  croptop "0.8633";
%cropright "0.6518";  cropbottom "0.7105";
%filename 'tranmf6.ps';file-properties "XNPEU";}} }%
%BeginExpansion
\begin{figure}
[p]
\begin{center}
\includegraphics[
trim=2.001846in 7.819479in 2.958586in 1.504465in,
height=1.7115in,
width=3.5639in
]%
{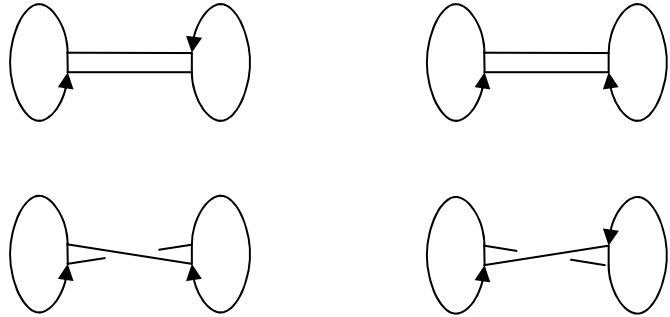}%
\caption{Disks and bands representing edges labeled $1$ and $-1$,
respectively.}%
\label{tranmf6}%
\end{center}
\end{figure}
%EndExpansion
%TCIMACRO{\FRAME{ftbpFU}{3.5924in}{4.612in}{0pt}{\Qcb{At the top, a ribbon
%graph $G$. In the middle, the disks and bands of $S(G)$. At the bottom, the
%medial graph $F(G)$.}}{\Qlb{tranmf7}}{tranmf7.ps}%
%width 3.5924in;  height 4.612in;  depth 0pt;  original-width 8.4968in;
%original-height 11.0056in;  cropleft "0.2514";  croptop "0.8901";
%cropright "0.8104";  cropbottom "0.3345";
%filename 'tranmf7.ps';file-properties "XNPEU";}} }%
%BeginExpansion
\begin{figure}
[ptb]
\begin{center}
\includegraphics[
trim=2.136096in 3.681373in 1.610993in 1.209515in,
height=4.612in,
width=3.5924in
]%
{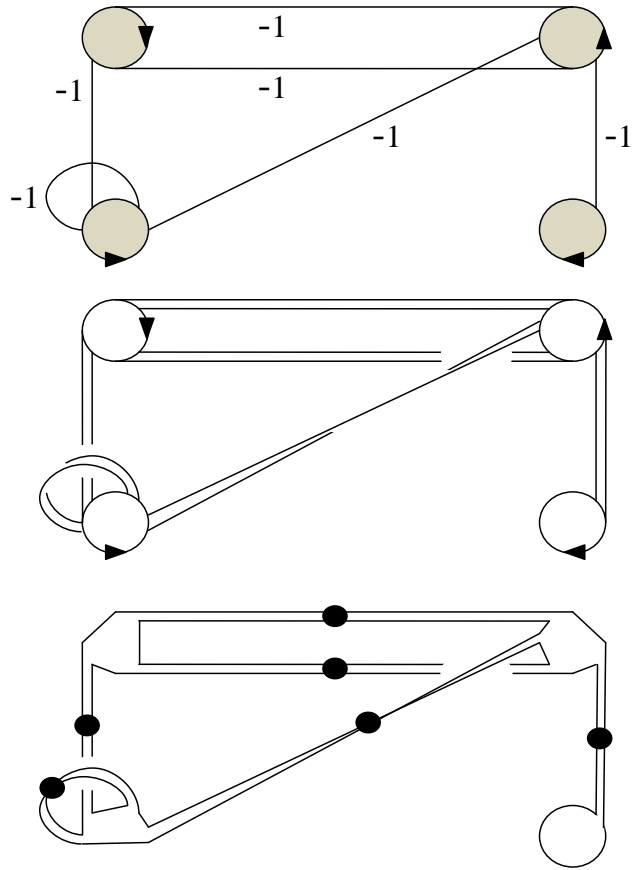}%
\caption{At the top, a ribbon graph $G$. In the middle, the disks and bands of
$S(G)$. At the bottom, the medial graph $F(G)$.}%
\label{tranmf7}%
\end{center}
\end{figure}
%EndExpansion

The purpose of this section is to present another description of equivalence
classes of ribbon graphs, using circuit partitions in 4-regular graphs. This
description appears at least implicitly in writings of Bouchet \cite{Bmaps},
Edmonds \cite{E} and Tutte \cite{Tu}, but it does not seem to have been used
in more recent work. Suppose a ribbon graph $G$ is given, with disks
$D_{1},...,D_{k}$ corresponding to the vertices of $G$ and bands
$B_{1},...,B_{n}$ corresponding to the edges of $G$. The medial graph $F=F(G)$
is a 4-regular graph constructed from $G$ as follows. First, $V(F)=E(G)$.
Second, the four half-edges of $F$ incident on a vertex $v_{i}$ correspond to
the four corners of the band $B_{i}$. Finally, the edges of $F$ are obtained
by pairing together the half-edges that correspond to consecutive points on a
disk boundary, but do not correspond to the same band-end. See Fig.
\ref{tranmf7} for an example.

There are two natural circuit partitions in the medial graph, which we denote
$\delta$ and $\varepsilon$; like any circuit partitions, they are determined
by choosing the appropriate transitions at each vertex. The $\delta$
transition at a vertex $v_{i}$ of $F(G)$ pairs together the half-edges
corresponding to the same end of the band $B_{i}$. We use the letter $\delta$
because the circuits of this partition correspond to the disks $D_{1}%
,...,D_{k}$. The $\varepsilon$ transition at $v_{i}$ pairs together the
half-edges corresponding to the same edge of the band $B_{i}$. The circuits of
the $\varepsilon$ partition correspond to the boundary curves of $S(G)$.

The $\delta$ and $\varepsilon$ circuit partitions of the example of Fig.
\ref{tranmf7} are indicated in\ Fig. \ref{tranmf8}. Circuits are indicated
using the convention that when a circuit is followed through a vertex, the
style (dashed or plain) is maintained; however it is sometimes necessary to
change the style in the middle of an edge, to make sure that the transitions
are indicated unambiguously.%
%TCIMACRO{\FRAME{ftbFU}{3.1903in}{3.365in}{0pt}{\Qcb{The $\delta$ and
%$\varepsilon$ circuit partitions of the example of Fig. \ref{tranmf7}.}%
%}{\Qlb{tranmf8}}{tranmf8.ps}{\special{ language "Scientific Word";
%type "GRAPHIC";  maintain-aspect-ratio TRUE;  display "USEDEF";
%valid_file "F";  width 3.1903in;  height 3.365in;  depth 0pt;
%original-width 8.4968in;  original-height 11.0056in;  cropleft "0.2824";
%croptop "0.9037";  cropright "0.7786";  cropbottom "0.4990";
%filename 'tranmf8.ps';file-properties "XNPEU";}} }%
%BeginExpansion
\begin{figure}
[tb]
\begin{center}
\includegraphics[
trim=2.399497in 5.491795in 1.881192in 1.059839in,
height=3.365in,
width=3.1903in
]%
{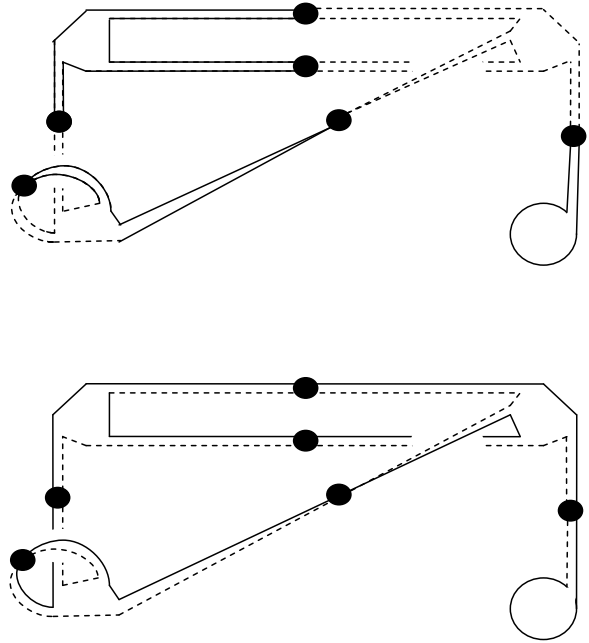}%
\caption{The $\delta$ and $\varepsilon$ circuit partitions of the example of
Fig. \ref{tranmf7}.}%
\label{tranmf8}%
\end{center}
\end{figure}
%EndExpansion

It is a simple matter to reverse the construction. Suppose a 4-regular graph
$F$ is given with two circuit partitions that do not involve the same
transition at any vertex. Label the two circuit partitions $\delta$ and
$\varepsilon$, and let $\{\Delta_{1},...,\Delta_{k}\}$ be the circuit
partition of $F$ determined by the $\delta$ transitions. Let $V(F)=\{v_{1}%
,...,v_{n}\}$ and for each $j\in\{1,...,k\}$, let $\Delta_{j}$ be $v_{j1}$,
$h_{j1}$, $h_{j1}^{\prime}$, $v_{j2}$, $h_{j2}$, $h_{j2}^{\prime}$, ...,
$h_{j(\ell_{j}-1)}^{\prime}$, $v_{j\ell_{j}}=v_{j1}$. Let $D_{1},...,D_{k}$ be
pairwise disjoint disks and for each $j\in\{1,...,k\}$, let $w_{j1}$,
$w_{j1}^{\prime}$, $w_{j2}$, $w_{j2}^{\prime}$, $...$, $w_{j(\ell_{j}-1)}$,
$w_{j(\ell_{j}-1)}^{\prime}$ be $2(\ell_{j}-1)$ distinct, consecutive points
on the boundary circle of $D_{j}$. Also, let $w_{j\ell_{j}}=w_{j1}$. Associate
a band $B_{i}$ with each vertex $v_{i}$ of $F$; if $v_{i}=v_{ab}=v_{cd}$ then
one end of $B_{i}$ is the arc from $w_{a(b-1)}^{\prime}$ to $w_{ab}$ on
$D_{a}$, and the other end of $B_{i}$ is the arc from $w_{c(d-1)}^{\prime}$ to
$w_{cd}$ on $D_{c}$. The $\delta$ and $\varepsilon$ circuit partitions do not
involve the same transition at any vertex, so the $\varepsilon$ transition at
$v_{i}$ pairs $h_{a(b-1)}^{\prime}$ with one of $h_{c(d-1)}^{\prime},h_{cd}$,
and pairs $h_{ab}$ with the other one of $h_{c(d-1)}^{\prime},h_{cd}$. One
edge of the band $B_{i}$ should connect $w_{a(b-1)}^{\prime}$ to one of
$w_{c(d-1)}^{\prime},w_{cd}$ and the other edge of $B_{i}$ should connect
$w_{ab}$ to the other one of $w_{c(d-1)}^{\prime},w_{cd}$, as dictated by the
$\varepsilon$ transition.

There are several reasons that representing a ribbon graph $G$ using circuit
partitions in the medial graph $F(G)$ is of value.

1. Clearly if $G$ and $G^{\prime}$ are ribbon graphs with the same medial,
then they differ only in their choices of which of the three transitions at
each vertex are designated $\delta$ and $\varepsilon$. (No particular
designation is used for the third transition.) This observation provides a new
explanation of Chmutov's theory of partial duality \cite{Ch} and
Ellis-Monaghan's and Moffatt's more general twisted duality \cite{EMM, EMM2}.
Namely: partial duals are obtained by interchanging the $(\delta,\varepsilon)$
designations at some vertices, and twisted duals are obtained by permuting the
$(\delta,\varepsilon,$ other) designations at some vertices. A distinctive
feature of these\ duality theories is their reliance on surface geometry; this
reliance can be problematic because twisted duals of ribbon graphs are not
naturally embedded in the same surface. They are, however, represented by
different pairs of circuit partitions\ in the same medial graph.

2. In particular, the circuit partition representation helps one to understand
a theorem of Ellis-Monaghan and Moffatt \cite{EMM}, that two ribbon graphs are
twisted duals if and only if they have the same medial graph. Their proof
seems to require embedding the ribbon graphs, constructing embedded medials,
and then \textquotedblleft disembedding\textquotedblright\ the medials. The
circuit partition representation gives an alternative way to think about this theorem.

3. As we discuss in the next section, topological properties of the surface
$S(G)$ are represented by combinatorial properties of the medial graph $F(G)$
and transition matroid $M_{\tau}(F(G))$.

Before closing this section, we recall a result of Jaeger \cite{J3} that was
mentioned in Section 2: every cographic matroid is represented by a square
$GF(2)$-matrix whose off-diagonal entries agree with some interlacement matrix
$\mathcal{A}(\mathcal{I}(C))$. The discussion above contains the following
outline of a proof of Jaeger's theorem. If $G$ is a graph then we may consider
it as a ribbon graph, with any rotation system. $G$ is then isomorphic to the
touch-graph of the circuit partition of $F(G)$ given by the $\delta$
transitions. As discussed in Section 5, the dual of the cycle matroid of $G$
is then isomorphic to the submatroid of $M_{\tau}(F(G))$ consisting of the
various elements $\delta(v)$, $v\in V(F(G))$. This outline for a proof of
Jaeger's theorem is detailed in \cite{BHT}, without any explicit reference to
topological graph theory. J. A. Ellis-Monaghan noticed it, and was kind enough
to suggest that the first part of the argument might be modified to involve
ribbon graphs. The discussion above verifies her insight.

\section{Topological Tutte polynomials}

Several authors have studied polynomials associated with a ribbon graph $G$,
which combine combinatorial information regarding $G$ and topological
information regarding $S(G)$. Two such polynomials were introduced in the
1970s by Las Vergnas \cite{L2} and Penrose \cite{Pen}, but broader interest in
such polynomials seems to have been relatively quiet until the work of
Bollob\'{a}s and Riordan \cite{BR1, BR2}. (In contrast, the intervening
decades were a time of very intense research regarding knot polynomials.) In
particular, the discussion of duality in \cite{BR2} has stimulated several
interesting developments, including the geometrically inspired partial duality
of Chmutov \cite{Ch} and twisted duality of Ellis-Monaghan and Moffatt
\cite{EMM, EMM2}. We do not attempt to provide a detailed account of these
topics here. Instead we content ourselves with the obvious observation that in
the representation of ribbon graphs using circuit partitions discussed in the
previous section, the (geometric) dual ribbon graph of Bollob\'{a}s and
Riordan \cite{BR2} is represented by interchanging the $\delta$ and
$\varepsilon$ transitions at all vertices of $F(G)$, the dual with respect to
an edge of Chmutov \cite{Ch} is represented by interchanging the $\delta$ and
$\varepsilon$ transitions at a single vertex of $F(G)$, and the half-twist of
an edge of Ellis-Monaghan and Moffatt \cite{EMM} is represented by
interchanging the $\varepsilon$ transition at a vertex with the non-$\delta$,
non-$\varepsilon$ transition.

As discussed in the introduction, Theorem \ref{tranmat} implies that any
polynomial which has a description as a vertex-weighted Martin polynomial of a
4-regular graph $F$ can be derived from a parametrized Tutte polynomial of
$M_{\tau}(F)$. Theorems \ref{MTutte} and \ref{JTutte} follow directly, because
the polynomials mentioned in these theorems have such descriptions. For
instance, the homflypt polynomial \cite{FYHLMO, PT} has a circuit partition
model due to Jaeger \cite{J4}, which has been extended to the Kauffman
polynomial \cite{KI} by Kauffman \cite{KT}.

At first glance, topological Tutte polynomials like those of Bollob\'{a}s and
Riordan \cite{BR1, BR2} may seem to fall into a different category, because
their definitions do not involve circuit partitions. Instead, the topological
Tutte polynomials of a ribbon graph $G$ are defined using combinatorial
properties of $G$ and topological properties of $S(G)$. It turns out, though,
that almost all of this information is available in the transition matroid of
$F(G)$.

\begin{proposition}
\label{tranrib}Let $G$ be a ribbon graph. For each $v\in V(F(G))$, let
$\delta(v)$ and $\varepsilon(v)$ denote the $\delta$ and $\varepsilon$
transitions at $v$.

\begin{enumerate}
\item $G$, the surface $S(G)$ and the medial $F(G)$ all have the same number
of connected components.

\item The boundary curves of $S(G)$ correspond to the circuits in the
$\varepsilon$ partition.

\item The cycle matroid of $G$ is the dual matroid of $M_{\tau}(\delta).$

\item For each subset $A\subseteq V(F(G))$, let $P_{A}$ be the circuit
partition of $F(G)$ that involves $\delta$ transitions at vertices in $A$, and
$\varepsilon$ transitions at vertices not in $A$. Then $S(G)$ is orientable if
and only if $\left\vert P_{A}\right\vert \neq\left\vert P_{A-\{a\}}\right\vert
$ $\forall A\subseteq V(F(G))$ $\forall a\in A$.
\end{enumerate}
\end{proposition}

\begin{proof}
For part 1, note that two vertices are connected by a walk in $G$ if and only
if the corresponding disks are connected by a sequence of disks and bands in
$S(G)$, and such a sequence of disks and bands corresponds to a walk in
$F(G)$. The walk in $F(G)$ may be considerably longer than the original walk
in $G$, of course, because vertices of $G$ (which may be traversed in a single
step) are replaced with segments of the $\delta$ circuits of $F(G)$.

Part 2 is clear, as the circuits of the $\varepsilon$ partition are defined to
follow the boundary of $S(G)$.

Part 3 is discussed in\ Section 5.

For part 4, consider that $S(G)$ is orientable if and only if it is possible
to coherently orient the boundaries of all the disks and bands involved in its
construction. Clearly this holds if and only if it is possible to choose
edge-directions in $F(G)$ that are respected by all the $\delta$ and
$\varepsilon$ transitions. If it is possible to choose such edge-directions,
then it must be that $\left\vert P_{A}\right\vert \neq\left\vert
P_{A-\{a\}}\right\vert $ for every choice of $A\subseteq V(F(G))$ and $a\in
A$. See Fig. \ref{tranmf9}, where the circuit partition pictured on the left
is one of $P_{A},P_{A-\{a\}}$ and the circuit partition pictured in the middle
is the other of $P_{A},P_{A-\{a\}}$; clearly the partition pictured in the
middle includes one more circuit.

For the converse, recall Kotzig's observation that $F$ must have an Euler
system $C$ that involves only $\delta$ and $\varepsilon$ transitions \cite{K}.
The recursive construction of such an Euler system is simple. Begin with any
circuit partition $P_{0}$ that involves only $\delta$ and $\varepsilon$
transitions. If $P_{0}$ is not an Euler system, there is a vertex $v$ at which
two distinct circuits of $P_{0}$ are incident. The transition of $P_{0}$ at
$v$ is either $\delta(v)$ or $\varepsilon(v)$; let $P_{1}$ be the circuit
partition that involves the same transition as $P_{0}$ at every other vertex,
and involves the different element of $\{\delta(v),\varepsilon(v)\}$. Then
$\left\vert P_{1}\right\vert =\left\vert P_{0}\right\vert -1$, and $P_{1}$
inherits the property that it involves only $\delta$ and $\varepsilon$
transitions. Repeat this process until an Euler system $C$ is obtained.

As $C$ involves only $\delta$ and $\varepsilon$ transitions, $C=P_{B}$ for
some $B\subseteq V(F(G))$. Choose orientations for the circuits in $C$, and
use these orientations to direct all the edges of $F(G)$. If $b\in B$ and
$\varepsilon(b)$ is inconsistent with these edge-directions then Fig.
\ref{tranmf9} indicates that $\left\vert P_{B}\right\vert =\left\vert
P_{B-\{b\}}\right\vert $. If $v\notin B$ and $\delta(v)$ is inconsistent with
these edge-directions then Fig. \ref{tranmf9} indicates that $\left\vert
P_{B}\right\vert =\left\vert P_{B\cup\{v\}}\right\vert $. Consequently if
$\left\vert P_{A}\right\vert \neq\left\vert P_{A-\{a\}}\right\vert $ $\forall
A\subseteq V(F(G))$ $\forall a\in A$ then the edge-directions must be
respected by all the $\delta$ and $\varepsilon$ transitions.
\end{proof}

%

%TCIMACRO{\FRAME{ftbpFU}{4.7997in}{0.6166in}{0pt}{\Qcb{Three circuit partitions
%that differ at only one vertex. The one on the right is inconsistent with the
%indicated edge-directions.}}{\Qlb{tranmf9}}{tranmf9.ps}%
%width 4.7997in;  height 0.6166in;  depth 0pt;  original-width 8.4968in;
%original-height 11.0056in;  cropleft "0.1730";  croptop "0.8844";
%cropright "0.9213";  cropbottom "0.8131";
%filename 'tranmf9.ps';file-properties "XNPEU";}} }%
%BeginExpansion
\begin{figure}
[ptb]
\begin{center}
\includegraphics[
trim=1.469946in 8.948653in 0.668698in 1.272247in,
height=0.6166in,
width=4.7997in
]%
{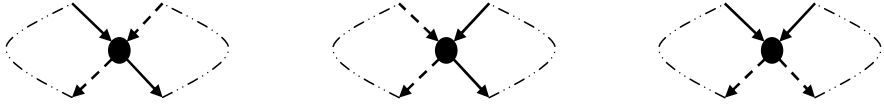}%
\caption{Three circuit partitions that differ at only one vertex. The one on
the right is inconsistent with the indicated edge-directions.}%
\label{tranmf9}%
\end{center}
\end{figure}
%EndExpansion

As the circuit-nullity formula ties the number of circuits in a circuit
partition $P$ to the rank of $\tau(P)$ in the matroid $M_{\tau}(F(G))$, it
follows that $M_{\tau}(F(G))$ contains enough information to determine the
cycle matroid of $G$,\ the number of boundary curves in $S(G)$, and the
orientability of $S(G)$. The matroid cannot detect the number of connected
components, of course, because it is invariant under connected sums and
separations. Theorem \ref{BR} follows, for the various topological Tutte
polynomials are all defined by subset expansions in which the contribution of
a subset $X\subseteq E(G)$ is determined by the cycle matroid of $G$, the
cycle matroid of the geometric dual, and topological characteristics of the
surface $S(G[X])$, where $G[X]$ is the graph obtained from $G$ by removing
edges not in $X$. When we remove an edge from $G$ the effect on the medial
$F(G)$ is simply to perform the detachment corresponding to the $\delta$
transition at the corresponding vertex. As discussed in\ Section 2, the
transition matroid of the detached graph is a minor of $M_{\tau}(F(G))$, so
all information about that matroid is present in $M_{\tau}(F(G))$.

We should note that all the information about $G$ and $S(G)$ discussed in
Proposition \ref{tranrib} is contained in the submatroid of $M_{\tau}(F(G))$
consisting of the $\delta$ and $\varepsilon$ transitions. The Euler systems of
$F(G)$ involving only $\delta$ and $\varepsilon$ transitions are determined by
this submatroid, and these Euler systems in turn determine a $\Delta$-matroid.
The corresponding $\Delta$-matroid version of Proposition \ref{tranrib} is due
to Bouchet \cite{Bmaps}, and Proposition \ref{tranrib} could be deduced from
the $\Delta$-matroid version simply by observing that the matroid $M_{\tau
}(F(G))$ contains enough information to determine all the $\Delta$-matroids
associated with ribbon graphs with $F(G)$ as medial. This observation is a
special case of the fact that all binary $\Delta$-matroids are determined by
isotropic matroids of graphs; see \cite{Tnewnew} for details.

\section{Planar 4-regular graphs}

If a 4-regular graph $F$ is imbedded in the plane, then the complementary
regions can be colored black or white, in such a way that regions that share
an edge of $F$ have different colors. (This observation dates back to the
founding of topology\ \cite{List}.) Let $P_{B}$ and $P_{W}$ be the circuit
partitions of $F$ that give the boundaries of the black and white regions,
respectively. Euler's formula tells us that
\[
2c(F)=\left\vert V(F)\right\vert -\left\vert E(F)\right\vert +\left\vert
P_{B}\right\vert +\left\vert P_{W}\right\vert =-n+\left\vert P_{B}\right\vert
+\left\vert P_{W}\right\vert \text{,}%
\]
and the circuit-nullity formula tells us that
\[
\left\vert P_{B}\right\vert -c(F)=n-r(\tau(P_{B}))\text{ and }\left\vert
P_{W}\right\vert -c(F)=n-r(\tau(P_{W}))\text{.}%
\]
Combining these formulas we conclude that
\[
r(\tau(P_{B}))+r(\tau(P_{W}))=2n+2c(F)-\left\vert P_{B}\right\vert -\left\vert
P_{W}\right\vert =n\text{.}%
\]
That is, $P_{B}$ and $P_{W}$ satisfy the equivalent conditions of Corollary
\ref{Martin}.

Conversely, suppose $F$ is a 4-regular graph with circuit partitions $P_{1}$
and $P_{2}$, which satisfy the equivalent conditions of Corollary
\ref{Martin}. Let $G$ be the ribbon graph constructed from $F$ using $P_{1}$
and $P_{2}$ as the $\delta$ and $\varepsilon$ circuit partitions
(respectively), as in\ Section 7. $S(G)$ is a surface with boundary
constructed from $\left\vert P_{1}\right\vert $ disks and $n$ bands, which has
$\left\vert P_{2}\right\vert $ boundary curves. Consequently, if $S^{\prime
}(G)$ is the closed surface obtained from $S(G)$ by attaching a disk to each
boundary curve, then the Euler characteristic of $S^{\prime}(G)$ is
$\left\vert P_{1}\right\vert -n+\left\vert P_{2}\right\vert $. The
circuit-nullity formula tells us that%
\[
\left\vert P_{1}\right\vert -n+\left\vert P_{2}\right\vert =c(F)+n-r(\tau
(P_{1}))-n+c(F)+n-r(\tau(P_{2}))
\]
so since $r(\tau(P_{1}))+r(\tau(P_{2}))=n$, the Euler characteristic of
$S^{\prime}(G)$ is $2c(F)$. As $c(F)$ is the number of connected components of
$S^{\prime}(G)$, we conclude that each connected component of $S^{\prime}(G)$
has Euler characteristic 2; that is, each connected component is a sphere. $F$
is imbedded in $S^{\prime}(G)$, so $F$ is planar.

We have proven the following.

\begin{proposition}
\label{Martin3}A 4-regular graph $F$ is planar if and only if it has a pair of
circuit partitions that satisfy the equivalent conditions of Corollary
\ref{Martin}.
\end{proposition}

The intent of this proposition is not to provide a new criterion for
planarity, but merely to indicate that $M_{\tau}(F)$ incorporates information
connected to familiar planarity criteria. For instance, the reader will
certainly not be surprised that Corollary \ref{Martin} indicates a connection
between planarity and matroid duality. It takes only a little longer to see
that Corollary \ref{Martin} is also connected to the following planarity
criterion, which is part of several solutions of the Gauss crossing problem
that have appeared in the literature \cite{DFOM, GR, RR}.

\begin{corollary}
\label{bipartite}Let $F$ be a 4-regular graph. Then $F$ is planar if and only
if it has an Euler system whose interlacement graph is bipartite.
\end{corollary}

\begin{proof}
Suppose $F$ is planar, and let $P_{1}$ and $P_{2}$ be circuit partitions of
$F$ that satisfy the equivalent conditions of Corollary \ref{Martin}. For each
$v\in V(F)$, let $\tau_{1}(v)$ and $\tau_{2}(v)$ be the transitions involved
in $P_{1}$ and $P_{2}$, respectively. Let $B_{1}$ be a basis of $M_{\tau
}(P_{1})$, let $V_{1}=\{v\in V(F)\mid\tau_{1}(v)\in B_{1}\}$, and let
$V_{2}=V(F)-V_{1}$. According to condition 3 of Corollary \ref{Martin},
$B_{2}=\{\tau_{2}(v)\mid v\in V_{2}\}$ is a basis of $M_{\tau}(P_{2})$.
Condition 4 of Corollary \ref{Martin} then tells us that $B_{1}\cup B_{2}$ is
a basis of $M_{\tau}(F)$.

Let $C$ be the Euler system of $F$ with $\tau(C)=B_{1}\cup B_{2}$. Let
\[
\mathcal{A}(\mathcal{I}(C))=%
\begin{pmatrix}
A_{11} & A_{12}\\
A_{21} & A_{22}%
\end{pmatrix}
\text{,}%
\]
where the rows of $A_{ij}$ correspond to vertices from $V_{i}$ and the columns
of $A_{ij}$ correspond to vertices from $V_{j}$. By Definition \ref{matroid}
\[
r(M_{\tau}(P_{1}))=r%
\begin{pmatrix}
I_{1} & A_{12}\\
\mathbf{0} & A_{22}^{\prime}%
\end{pmatrix}
,
\]
where $I_{1}$ is a $\left\vert V_{1}\right\vert \times\left\vert
V_{1}\right\vert $ identity matrix and $A_{22}^{\prime}$ is obtained from
$A_{22}$ by placing a 1 at each diagonal entry corresponding to a vertex $v\in
V_{2}$ where $\tau_{1}(v)=\psi_{C}(v)$. As $r(M_{\tau}(P_{1}))=\left\vert
B_{1}\right\vert =\left\vert V_{1}\right\vert =r(I_{1})$, $A_{22}^{\prime}$
cannot have any nonzero entry. It follows that no two vertices of $V_{2}$ are
neighbors in $\mathcal{I}(C)$. Similarly,
\[
r(M_{\tau}(P_{2}))=r%
\begin{pmatrix}
A_{11}^{\prime} & \mathbf{0}\\
A_{21} & I_{2}%
\end{pmatrix}
,
\]
where $I_{2}$ is a $\left\vert V_{2}\right\vert \times\left\vert
V_{2}\right\vert $ identity matrix and $A_{11}^{\prime}$ agrees with $A_{11}$
off the diagonal. As $r(M_{\tau}(P_{2}))=\left\vert V_{2}\right\vert $,
$A_{11}^{\prime}$ cannot have any nonzero entry. It follows that no two
vertices of $V_{1}$ are neighbors in $\mathcal{I}(C)$, so $\mathcal{I}(C)$ is bipartite.

Suppose conversely that $F$ has an Euler system $C$ whose interlacement graph
is bipartite. Let $V(F)=V_{1}\cup V_{2}$, so that $V_{1}\cap V_{2}%
=\varnothing$ and every edge of $\mathcal{I}(C)$ connects a vertex from
$V_{1}$ to a vertex from $V_{2}$. Then
\[
\mathcal{A}(\mathcal{I}(C))=%
\begin{pmatrix}
\mathbf{0} & A_{12}\\
A_{21} & \mathbf{0}%
\end{pmatrix}
\text{,}%
\]
where $A_{12}$ and $A_{21}$ are transposes. For $i\in\{1,2\}$ let $P_{i}$ be
the circuit partition of $F$ with
\[
\tau(P_{i})=\{\phi_{C}(v)\mid v\in V_{i}\}\cup\{\chi_{C}(v)\mid v\notin
V_{i}\}\text{.}%
\]
Then Definition \ref{matroid} implies that
\[
r(M_{\tau}(P_{1}))=r%
\begin{pmatrix}
I_{1} & A_{12}\\
\mathbf{0} & \mathbf{0}%
\end{pmatrix}
\text{ and }r(M_{\tau}(P_{2}))=r%
\begin{pmatrix}
\mathbf{0} & \mathbf{0}\\
A_{21} & I_{2}%
\end{pmatrix}
\text{,}%
\]
where $I_{1}$ and $I_{2}$ are identity matrices. It follows that $r(M_{\tau
}(P_{1}))+r(M_{\tau}(P_{2}))=n$, so $P_{1}$ and $P_{2}$ satisfy the equivalent
conditions of Corollary \ref{Martin}.
\end{proof}

\section{Conclusion}

The transition matroids of 4-regular graphs provide unified descriptions of
several objects of graph theory and knot theory. One virtue of this unified
context is to allow the application to these objects of matroid techniques,
which have been well studied. Algorithmic and combinatorial properties of
parametrized Tutte polynomials of matroids \cite{BR, EMT, T, Told, Za} provide
activities descriptions, complexity results and substitution techniques that
apply directly to all the different polynomials mentioned in the introduction.
Much remains to be done, though, to provide details of these applications, and
to relate them to results already in the literature. For example, the
circuit-nullity formula tells us that some bases of transition matroids
correspond to Euler systems, so it seems natural to guess that matroidal basis
activities can be used to explain the significance for the
Bollob\'{a}s-Riordan polynomial of the quasi-trees introduced by Champanerkar,
Kofman and Stoltzfus \cite{CKT}. Similarly, it seems natural to guess that
formulas for parametrized Tutte polynomials of series/parallel extensions of
matroids can be used to explain the tangle substitution formula of Jin and
Zhang \cite{JZ} for the homflypt polynomial. Guessing that there are such
explanations is not the same as actually providing them, though.

\section{Dedication}

This paper is dedicated to the memory of Michel Las Vergnas, in appreciation
of the beauty and importance of his many contributions to the theory of graphs
and matroids. In particular, his papers on circuit partitions and embedded
graphs from the 1970s and 1980s \cite{L2, L1, L} have proven to be crucial to
the development of the theory of 4-regular graphs, and they inspire the
results presented here.

\end{document}